\documentclass[10pt,a4paper]{amsart}
\usepackage{amsaddr}
\usepackage[usenames]{color} 
\usepackage[applemac]{inputenc}
\usepackage{amssymb} 
\usepackage{amsmath} 
\usepackage{mathabx}
\usepackage{amsthm}
\usepackage[all]{xy}
\usepackage{enumerate}
\usepackage[T1]{fontenc}
\usepackage[left=3cm,right=3cm,top=3.7cm,bottom=3.5cm]{geometry}
\usepackage{cleveref}
\usepackage{graphics} 
\usepackage{graphicx}
\usepackage{pstricks,pst-node} 
\usepackage{tikz} 
\usepackage{xcolor}

\begin{document}
\theoremstyle{plain}
\newtheorem{deftn}{Definition}[section]
\newtheorem{lem}[deftn]{Lemma}
\newtheorem{prop}[deftn]{Proposition}
\newtheorem{thm}[deftn]{Theorem}
\newtheorem*{thm*}{Theorem}
\newtheorem*{prop*}{Proposition}
\newtheorem{cor}[deftn]{Corollary}
\newtheorem{conj}[deftn]{Conjecture}
\newtheorem{assump}{Assumption}[section]
\renewcommand{\theassump}{\Alph{assump}}
\makeatletter
\renewcommand{\thefigure}{\arabic{figure}}
\makeatother

\theoremstyle{definition}
\newtheorem{ex}[deftn]{Example}
\newtheorem{rk}[deftn]{Remark}

 
  \tikzset{math3d/.style=
 {x={(1.62cm,-0.27cm)},z={(2.67cm,0.57cm)},y={(0.9cm,2.27cm)}}}
 
  \tikzset{math2d/.style=
 {x={(1.1cm,-0.5cm)},z={(2.3cm,0.4cm)},y={(1cm,1.5cm)}}}

\newcommand{\A}{\mathcal{A}}
\newcommand{\B}{\mathcal{B}}
\newcommand{\C}{\mathcal{C}}
\newcommand{\CQ}{\mathcal{C}_Q}
\newcommand{\Cw}{\mathcal{C}_w}
\newcommand{\yjh}{\hat{y_j}}
\newcommand{\ylh}{\hat{y_1}}
\newcommand{\ynh}{\hat{y_n}}
\newcommand{\ykh}{\hat{y_k}}
\newcommand{\mjh}{\hat{\mu_j}}
\newcommand{\mkh}{\hat{\mu_k}}
\newcommand{\bmjh}{\hat{\boldsymbol{\mu}_j}}
\newcommand{\bmkh}{\hat{\boldsymbol{\mu}_k}}
\newcommand{\bmu}{\boldsymbol{\mu}}
\newcommand{\p}{\mathbb{P}}
\newcommand{\W}{{\bf W}}
\newcommand{\Aqnw}{\mathcal{A}_q (\mathfrak{n}(w))}
\newcommand{\Aqn}{\mathcal{A}_q (\mathfrak{n})}
\newcommand{\Yia}{Y_{i,a}}
\newcommand{\M}{{\bf M}}
\newcommand{\G}{{\bf G}}
\newcommand{\s}{\mathcal{S}}
\newcommand{\hs}{\mathcal{H}_{\mathcal{S}}}
\newcommand{\ps}{\mathcal{P}_{\mathcal{S}}}
\newcommand{\hw}{\mathcal{H}_{w}}
\newcommand{\pw}{\mathcal{P}_{w}}
\newcommand{\zn}{\mathbb{Z}^N}
\newcommand{\hyp}{\mathcal{H}}
\newcommand{\ds}{\Delta_{\mathcal{S}}}
\newcommand{\hgt}{\text{ht}}
\newcommand{\wt}{\text{wt}}
\newcommand{\Vect}{\text{Vect}}

 \begin{abstract}
 We show that for a finite-type Lie algebra $\mathfrak{g}$, the representation theory of quiver Hecke algebras provides a natural framework for the construction of Newton-Okounkov bodies associated to the quantum coordinate rings $\Aqnw$. When $\mathfrak{g}$ is simply-laced, we use Kang-Kashiwara-Kim-Oh's monoidal categorification to investigate the cluster theory of these bodies. In particular, our construction yields a simplex $\ds$ for every seed $\s$ of $\Aqnw$.
 We exhibit various properties of these simplices by characterizing their rational points, normal fans, and volumes. As an application, we prove an equality of rational functions relating Nakada's hook formula with the root partitions associated to cluster variables, suggesting further connections between cluster theory and the combinatorics of fully-commutative elements of Weyl groups. 
  \end{abstract}

\newcommand{\Address}{ \textsc{ \\ Universit\'e de Paris, \\  Institut de Math\'ematiques de Jussieu-Paris Rive Gauche, UMR 7586, \\ CNRS F-75013 Paris \\ FRANCE} \\
 E-mail: \texttt{elie.casbi@imj-prg.fr}}

\title[]{\Large{ {\bf NEWTON-OKOUNKOV BODIES FOR CATEGORIES OF MODULES OVER QUIVER HECKE ALGEBRAS}}}
 
   \author{Elie CASBI}

  \maketitle
 
     \date{}

\setcounter{tocdepth}{1}
\tableofcontents

\section{Introduction}

 Cluster algebras were introduced by Fomin-Zelevinsky \cite{FZ1} in order to study total positivity and canonical bases for quantum groups. They are commutative subalgebras of $\mathbb{Q}(x_1, \ldots , x_N)$ where $x_1, \ldots , x_N$ are some algebraically independent variables. More precisely, they are generated by a certain (finite or infinite) number of rational functions in $x_1, \ldots , x_N$, which are called \textit{cluster variables}. These cluster variables are defined inductively via a process called \textit{mutation}: the initial data consists in the variables $x_1,\ldots , x_N$ together with a quiver $Q$ (i.e. an oriented graph) with $N$ vertices. Such a data is called a \textit{seed}. The set of vertices of $Q$ splits into two parts: the \textit{unfrozen} part and the \textit{frozen} part. For every vertex $k$ in the unfrozen part, the mutation in the direction $k$ of the seed $((x_1, \ldots , x_N),Q)$ consists in producing a new variable $x'_k$ as well as a new quiver $Q'$. This new quiver has the same set of vertices than $Q$ but has different arrows. Thus one gets a new seed given by the variables $x_1, \ldots , x_{k-1}, x'_k , x_{k+1}, \ldots , x_N$ and the quiver $Q'$. In other words, a mutation in the direction $k$ replaces $x_k$ by $x'_k$ and leaves all the variables attached to the vertices $j \neq k$ unchanged. This process is involutive, i.e. the mutation in the direction $k$ transforms this new seed into the initial seed. Then one can iterate this procedure and all the new variables obtained from the initial seed after an arbitrary finite number of mutations are the cluster variables. Thus the data of an initial seed yields a unique (up to $\mathbb{Q}$-algebra isomorphism) cluster algebra, usually denoted by $\mathcal{A}((x_1, \ldots , x_N),Q)$. The monomials involving only cluster variables belonging to the same seed are called \textit{cluster monomials}.

 Here we are interested in a remarkable connection between cluster theory and representation theory involving the notion of \textit{monoidal categorification of cluster algebras} introduced by Hernandez-Leclerc \cite{HL}. Given a cluster algebra $\A$, a monoidal categorification of $\A$ is a monoidal category $\C$ such that there is a ring isomorphism $\A \longrightarrow K_0(\C)$ sending any cluster monomial onto the class of a simple object in $\C$. Here $K_0(\C)$ denotes the Grothendieck ring of $\C$. This is a strong assumption and is difficult to prove in general. The idea is to use results of cluster theory to understand the structure of the category $\C$. The first example of monoidal categorification appeared in \cite{HL}, where such a statement is proved for certain subcategories of finite-dimensional representations of quantum affine algebras in types $A_n$ and $D_4$. It was generalized to other simply-laced types by Nakajima \cite{Naka} using perverse sheaves on quiver varieties and Qin \cite{Qin}, and very recently by Kashiwara-Kim-Oh-Park \cite{KKOP21} who exhibited a large family of monoidal categorifications of cluster algebras via categories of modules over quantum affine algebras of arbitrary finite types, relying on former constructions due to Hernandez-Leclerc \cite{HLCrelle,HLJEMS}. In a different direction, Cautis-Williams \cite{CW} showed that the coherent Satake category of the loop Grassmannian is a monoidal categorification of certain affine coordinate rings. The works of Kang-Kashiwara-Kim-Oh \cite{KKK,KKKO3,KKKO} provided another class of examples of monoidal categorifications of cluster algebras via categories of modules over quiver Hecke algebras. Introduced by  Khovanov-Lauda \cite{KL} and Rouquier \cite{R}, quiver Hecke algebras are  $\mathbb{Z}$-graded algebras which categorify the negative part $U_q(\mathfrak{n})$ of the quantum group  $U_q(\mathfrak{g})$ (here $\mathfrak{g}$ is a symmetric Kac-Moody algebra).  Let $R-gmod$ denote the monoidal category of finite-dimensional graded modules over the quiver Hecke algebra corresponding to $\mathfrak{g}$. It was shown by Rouquier \cite{R} and Varagnolo-Vasserot \cite{VV} that the elements of the dual canonical basis of $U_q(\mathfrak{n})$ are in  bijection with the set of isomorphism classes of simple objects in $R-gmod$. The results of \cite{KKKO} show that the cluster structures on the quantum coordinate rings $\Aqnw$ (known from \cite{GLS}) admit monoidal categorifications via subcategories $\Cw$ of $R-gmod$. Here $w$ runs over the Weyl group $W$ of $\mathfrak{g}$. 
 
 In \cite{Casbi}, we showed that these monoidal categorification statements imply certain relationships between natural orderings of different natures.
  On the one hand, the cluster-theoretic \textit{dominance order} was introduced by F. Qin \cite{Qin}. This order depends on the choice of a seed: for any seed $\s$ in a cluster algebra $\A$, the dominance order for $\s$, denoted  $\preccurlyeq_{\s}$, is a partial ordering on the set of Laurent monomials in the cluster variables of $\s$. This order was used in \cite{Qin} to study certain bases in $\A$, called common triangular bases. 
 On the other hand, we can consider orderings of representation-theoretic nature, arising from parametrizations of simple objects in a monoidal categorification $\C$ of $\A$. In the case of $R-gmod$ or $\Cw$, simple modules are parametrized by  \textit{dominant words} (or \textit{root partitions}): choose a labeling $I$ of the simple roots of $\mathfrak{g}$ and fix a total order $<$ on $I$. We still denote $<$ the induced lexicographic order on the set $\mathcal{M}$ of all (finite) words on the alphabet $I$. 
    Then there exists a finite subset $\mathcal{GL}$ of $\mathcal{M}$ in bijection with the positive roots of $\mathfrak{g}$ such that the set of dominant words is
      $$ \M := \{ {\bf j}_1 \cdots {\bf j}_k \mid {\bf j}_1 , \ldots , {\bf j}_k \in \mathcal{GL} , {\bf j}_1 \geq \cdots \geq {\bf j}_k \} . $$
     The bijection between $\mathcal{GL}$ and $\Phi_{+}$ makes $<$ into an ordering on $\Phi_{+}$. Fix $w \in W$ and consider the unique reduced expression ${\bf w} = (i_1, \ldots , i_N)$ of $w$ corresponding to the restriction of $<$ to $\Phi_{+}^{w}$. Let $\mathcal{S}^{\bf w}$ denote the seed of $\Aqnw$ corresponding to ${\bf w}$ via {{\cite[Theorem 11.2.2]{KKKO}}}.
 In the case of a type $A_n$ underlying Lie algebra $\mathfrak{g}$ with $<$ being the natural order on $I = \{1 , \ldots , n \}$ and $w=w_0$ the longest element of $W$, a compatibility between $<$ and $\preccurlyeq_{\mathcal{S}^{\bf w_0}}$ was exhibited in \cite{Casbi}. We showed ({{\cite[Theorem 6.2]{Casbi}}}) that this compatibility implies certain relationships between dominant words and $g$-vectors. 
 Recently, Kashiwara-Kim \cite{KK} related the dominance orders corresponding to the seeds of $\Aqnw$ with the homogeneous degrees of certain renormalized $R$-matrices in $\Cw$. For every seed $\s$, they also constructed natural maps ${\bf g}_{\mathcal{S}}^{L}$ and ${\bf g}_{\mathcal{S}}^{R}$ from the set of classes of self-dual simple objects in $\Cw$ into $\mathbb{Z}^{l(w)}$. Using tropical transformation, they showed that these maps are bijective (after localizing the frozen variables) and indeed associate a cluster variable with its $g$-vector.  Moreover, in the case where $\s$ is one of the seeds given by {{\cite[Theorem 11.2.2]{KKKO}}},  they relate ${\bf g}_{\mathcal{S}}^{R}(M)$ to the cuspidal decomposition of $M$ ({{\cite[Proposition 3.14]{KK}}}).
 In Section~\ref{seedsection},  we use these results to give an explicit description of the seed $\mathcal{S}^{{\bf w}}$ in terms of dominant words for an underlying Lie algebra $\mathfrak{g}$ of type $A_n, D_n$ or $E_n$ and an arbitrary order $<$ on $I$.

 \begin{thm*}[Theorem~\ref{initpara}]
 Fix $w \in W$ and a total order $<$ on $I$. Denote by ${\bf w}$ the corresponding reduced expression of $w$ and let $(x_1, \ldots , x_N)$ be the cluster variables of the seed $\mathcal{S}^{{\bf w}}$ of $\Aqnw$. Let $\mu_k$ denote the dominant word such that $x_k=[L(\mu_k)]$ for every $k \in J$. 
 Write the canonical factorization of $\mu_k$ as 
 $$ \mu_k = ({\bf i}_{N})^{c_N} \cdots  ({\bf i}_{1})^{c_1} . $$
 Then the tuple  $(c_1, \ldots , c_N)$ is given by 
 $$ c_j =  \begin{cases}
        1  \quad  & \text{if $j \leq k$ and $i_j=i_k$} \\
        0 \quad &\text{otherwise.}
        \end{cases} $$
 \end{thm*}
 
This statement is thus a generalization of {{\cite[Theorem 6.1]{Casbi}}} where it was proved for $\mathfrak{g}$ of type $A_n$, $I = \{1, \ldots , n \}$ ordered with respect to the natural ordering and $w=w_0$ is the longest element of the $W$ (in this case $\Cw$ is the whole category $R-gmod$).
  As a consequence of this Theorem, the lexicographic ordering $<$ on $\mathcal{GL}$ induces a total ordering on the set of cluster variables of every seed in $\Aqnw$. Such orderings play a key role in the recent work of Rietsch-Williams \cite{RW} (see {{\cite[Section 7]{RW}}}) and hence provide us with a natural motivation for the construction of Newton-Okounkov bodies using the monoidal categorification framework.

  Newton-Okounkov bodies were introduced by Kaveh-Khovanskii \cite{KKh} and Lazarsfeld-Mustata \cite{LM} following an idea of Okounkov \cite{Oko}. Their construction requires the following elementary objects. Let $\A$ be a graded algebra over an algebraically closed  field ${\bf k}$, i.e. one has a decomposition $\A = \bigoplus_{n \geq 0} {\A}_n $ where the $\A_n$ are vector subspaces of $\A$ satisfying  ${\A}_n {\A}_m \subset {\A}_{n+m} $ for any $n,m \geq 0$ with  $ {\A}_0 = {\bf k}$. Assume we have a \textit{valuation}, i.e. a map $\Psi$  from $\A$ to $\zn$ equipped with a total ordering satisfying certain axioms (see Definition~\ref{defvalu}). Under certain assumptions (see Section~\ref{NObodies}), one has $ \dim {\A}_n = \sharp \Psi({\A}_n \setminus \{0\})$ and thus this valuation provides a useful tool to study the asymptotics of $\dim {\A}_n$ as $n \rightarrow \infty$. The Newton-Okounkov body associated to $\A$ is defined as 
 $$  \Delta(\A) := \overline{\text{ConvexHull} \left( \bigcup_{n \geq 1}  \frac{1}{n} \Psi({\A}_n \setminus \{0\}) \right)}. $$
It is a convex compact set but not a polytope in general. The main property of these bodies is that their volume is closely related to the asymptotic behaviour of the Hilbert function of $\A$ (see e.g. {{\cite[Theorem 2.31]{KKh}}}). The theory of Newton-Okounkov bodies has found applications in various areas, in particular in algebraic geometry \cite{BC} and representation theory \cite{BossPhD,LittFang,OF,Kaveh}. Recently,  Newton-Okounkov bodies appeared in a different context in the work of Rietsch-Williams \cite{RW}. In this case, Newton-Okounkov bodies are associated to certain algebras of global sections of line bundles over the Grassmannian $Gr_{n-k}(\mathbb{C}^n)$ (the variety of codimension $k$ vector subspaces in $\mathbb{C}^n$). The valuation is defined from the cluster algebra structure on the ring of functions on $Gr_{n-k}(\mathbb{C}^n)$.

In this paper we show that the representation theory of quiver Hecke algebras associated with a finite type simple Lie algebra $\mathfrak{g}$ provides a natural setting to construct Newton-Okounkov bodies.  We prove that the parametrizations of simple objects in $\Cw$ in terms of dominant words provide natural valuations $\Psi_w$ from $\Aqnw$ to $\mathbb{Z}^{l(w)}$ (Section~\ref{valudef}). There is also a natural choice of grading $x \longmapsto |x| \in \mathbb{N}$  on $\Aqnw$ coming from the structure of the category $\Cw$ (Section~\ref{grading}). This allows us to associate Newton-Okounkov bodies to the algebras $\Aqnw$ and more generally to any (graded) subalgebra of $\Aqnw$ (Section~\ref{NOCw}). Two kind of subalgebras will be of particular interest for us: the algebras $\Aqnw$ themselves, and the (free) subalgebras generated by the cluster variables of a same seed in $\Aqnw$.

Assume $\mathfrak{g}$ is simply-laced and let $\s = ((x_1, \ldots , x_N),B)$ be a seed in $\Aqnw$. The Newton-Okounkov body $\Delta(\Aqnw)$ is a $N-1$ dimensional simplex inside an affine hyperplane $\hyp$ in $\mathbb{R}^N$ (see Section~\ref{hyperplane}). The Newton-Okounkov body $\ds$ is also a simplex, contained in $\Delta(\Aqnw)$. In Sections~\ref{sectionds} and ~\ref{normalfan}, we prove the following properties of the simplices $\ds$:

 \begin{prop*}[cf. Proposition~\ref{propds}]
   The rational points of $\ds$ correspond to the cluster monomials for $\s$ in the following sense:
$$ \text{$x$ is a cluster monomial for $\s$} \enspace \Leftrightarrow \enspace \frac{1}{|x|} \Psi(x) \in \ds . $$
   Moreover, any rational point in $\ds$ is of the form $\frac{1}{|x|} \Psi(x)$ for some cluster monomial $x$ in $\Aqnw$. 
   \end{prop*}
   
   \smallskip
   
   \begin{thm*}[cf. Theorem~\ref{mainthm}]
    The normal fan of the simplex $\ds$ can be explicitly related to the dominance order for $\s$. More precisely, let $\overrightarrow{\mathcal{N}^{\s}}$ be the linear cone such that for any simple $M$ in $\Cw$ one has
 $$ \Psi([M]) + \overrightarrow{\mathcal{N}^{\s}} = \{ \Psi([N]) , [N] \preccurlyeq_{\s} [M] \} . $$
   There exists a unique explicit universal transformation $T \in \mathcal{M}_N(\mathbb{Q})$ (by universal we mean independent of the seed), such that for any seed $\s$, the cone $T \overrightarrow{\mathcal{N}^{\s}}$ is a face of the normal fan of $\ds$. The explicit expression of $T$ is given by Corollary~\ref{cordominance}.
 \end{thm*}

 This statement provides a natural geometric interpretation of the \textit{generalized parameters} $\mjh$ introduced in {{\cite[Section 4.2]{Casbi}}}. 
The proof is based on considering an \textit{epsilon tropical mutation} in the sense of \cite{Nakanishi} satisfied by the rays of the normal fan of $\ds$. This mutation relation is strongly related to the dominance order for $\s$. Then the desired transformation $T$ can be explicitly known from Theorem~\ref{initpara}. 
 Note that under the enlightment of \cite{KK}, one can think about dominant words as certain $g$-vectors and thus this statement can be seen as an analog of the well-known duality between $c$-vectors and $g$-vectors (see for instance \cite{FG,HPS}). However in the purely cluster-theoretic setting, one always considers $g$-vectors and $c$-vectors \textit{with respect to an initial seed}, whereas our statement is of a more representation-theoretic nature and does not rely on such a choice.

  In Section~\ref{hooksection}, we focus on the case of finite type cluster algebras. In this case, the simplices $\ds$ cover all of the simplex $\Delta(\Aqnw)$. Using Theorem~\ref{initpara}, we prove the following equality of rational functions:
  
  \begin{thm*}[cf. Theorem~\ref{prophook}]
Assume $w \in W$ is such that the cluster algebra $\Aqnw$ is of finite type. Then 
\begin{equation} \label{eqnintro}
 \prod_{\beta \in \Phi_{+}^{w}} \frac{1}{\beta} = \sum_{\s} \prod_{1 \leq j \leq N} \frac{1}{\beta_j^{\s}} . 
 \end{equation}
\end{thm*}

Here $\beta_1^{\s} , \ldots , \beta_N^{\s}$ denote the weights of the cluster variables of every seed $\s$. In \cite{Nakada}, Nakada proved a \textit{colored hook formula} that has a very similar form. However the underlying combinatorics is a priori very different from cluster theory. We end the section by discussing an example. We also note that possibly up to some universal constant, the formula obtained for $Vol(\Delta(\Aqnw))$ is the Peterson-Proctor hook formula (see for instance {{\cite[Section 8.1]{KR}}}) whereas $Vol(\ds)$ is similar but involves the heights of the cluster variables of $\s$ instead of the positive roots in $\Delta_{+}^{w}$.

 The paper is organized as follows: we begin in Section~\ref{background} with some reminders on the representation theory of quiver Hecke algebras (Section~\ref{remindKLR}) especially in the context of monoidal categorification of cluster algebras, following \cite{KKKO,KK} (Section~\ref{remindKK}). We also recall (Section~\ref{remindCasbi}) several basic objects introduced in \cite{Casbi}.
  In Section~\ref{seedsection}, we use the recent results of Kashiwara-Kim \cite{KK} to prove the main result of this paper (Theorem~\ref{initpara}).
   In Section~\ref{NObodyKLR} we construct certain Newton-Okounkov bodies in a natural way using the setting of the previous sections.
   Section~\ref{sectionds} is devoted to exhibiting some properties of the Newton-Okounkov bodies corresponding to monoidal seeds via this construction. Then we study the behaviour of these bodies under cluster mutation. The main tool is the tropical $\epsilon$-mutation in the sense of \cite{Nakanishi} (see Section~\ref{epsilonmut}). We show in Section~\ref{normalfan} that the normal fan of these bodies  essentially describes Qin's dominance order.
   Finally in Section~\ref{hooksection} we prove Theorem~\ref{prophook} and conclude with some possible connections with the combinatorics of colored hook formulas studied in \cite{Nakada}. 
 
  \bigskip
  
  \textit{The author is supported by the European Research Council under the European Union's Framework Programme H2020 with ERC Grant Agreement number 647353 Qaffine.}

  \subsection*{Acknowledgements}

  First of all I would like to thank my advisor David Hernandez for his constant support and encouragements as well as for many useful conversations. I would also like to thank S\'ebastien Boucksom for answering my questions about Newton-Okounkov bodies. I also thank Vincent Pilaud, Lara Bossinger and Joel Kamnitzer for very inspiring discussions. Finally I thank Hironori Oya for making valuable remarks on the first version of this work.

  \section{Quiver Hecke algebras and monoidal categorifications of quantum coordinate rings}
  \label{background}
  
 In this section we recall some representation-theoretic background and we fix notations. We begin with some reminders about Kleshchev-Ram's classification of finite-dimensional irreducible representations of finite type quiver Hecke algebras \cite{KR}. Then we recall how quiver Hecke algebras provide a useful framework for monoidal categorifications of cluster algebras following \cite{KKKO} and more recently \cite{KK}. Finally we recall several technical tools from \cite{Casbi} that will be  useful in the next sections. 
  
\subsection{General reminders on quiver Hecke algebras} 
\label{remindKLR}

 Let $\mathfrak{g}$ be a semisimple Lie algebra of finite type, $I$ the set of vertices of the Dynkin diagram of $\mathfrak{g}$. We use the following standard Lie-theoretic notations: $\Pi = \{ \alpha_i , i \in  I \}$ stands for the set of simple roots, $Q_{+} := \bigoplus_{i \in I} \mathbb{N} \alpha_i$, and $\Phi_{+}$ denotes the set of positive roots. 
 We also let $\mathcal{M}$ denote the set of (finite) words over the alphabet $I$. For $\nu = h_1, \ldots , h_r  \in \mathcal{M}$, we define the \textit{weight} of $\nu$ as the element of $Q_{+}$ given by 
  $$ \wt(\nu) := \sum_{i \in I}  \sharp \{k, h_k = i \} \alpha_i . $$
  To any $\beta \in Q_{+}$ one associates a $\mathbb{Z}$-graded associative algebra $R(\beta)$ defined by generators and relations. We refer to \cite{KL,KR,KKK} for precise definitions. Let us only outline the fact that among the generators of $R(\beta)$, one has a family of idempotents $ \{ e(\nu) , \text{$\nu \in \mathcal{M}$ such that $\wt(\nu) = \beta$} \}$, satisfying the relations 
  $$e(\mu)e(\nu) = \delta_{\mu,\nu} e(\nu) . $$
   This family of algebras is called quiver Hecke algebras. For any $\beta \in Q_{+}$, one denotes by $R(\beta)$-gmod the category of finite-dimensional graded $R(\beta)$-modules. One also sets 
  $$ R-gmod := \bigoplus_{\beta} R(\beta)-gmod . $$
  The main property of quiver Hecke algebras is that the category $R-gmod$ categorifies the quantum coordinate ring $\Aqn$ (which is isomorphic to the positive part of the quantum group $U_q(\mathfrak{g})$) in a way that sends the isomorphism classes of simple objects in $R-gmod$ bijectively onto the elements of the dual canonical basis of $\Aqn$. 
  
  The construction of irreducible finite-dimensional representations over quiver Hecke algebras of finite type was done by Kleshchev-Ram \cite{KR}. This parametrization uses the combinatorics of Lyndon words, or root partitions. It has been generalized by Kleshchev \cite{Kle} and McNamara \cite{McN} to affine type quiver Hecke algebras. 
 Recall that $\mathcal{M}$ denotes the set of finite words over the alphabet $I$. Fix a total order $<$ on $I$; thus $\mathcal{M}$ is totally ordered for the induced lexicographic order $\leq$. 
  For every $\beta \in Q_{+}$, any finite-dimensional $R(\beta)$-module $V$ decomposes as 
$$ V = \bigoplus_{\nu, \wt(\nu)=\beta} e(\nu) \cdot V . $$
The subspace $e(\nu) \cdot V$ can be seen as some kind of weight space by analogy with the representation theory of semisimple finite-dimensional Lie algebras. 
Hence one can consider the \textit{highest word} of $V$, i.e. the biggest $\nu$ (for the total order $\leq$) such that $e(\nu) \cdot V$ is non zero. 
   We set 
$$ \M := \{ \nu \in \mathcal{M} \mid \exists V \in R(\wt(\nu))-mod , \text{$\nu$ is the highest word of $V$}  \}. $$
The following statement is the main result of \cite{KR} and shows that $\M$ is in bijection with the set of isomorphism classes of simple modules in $R-gmod$. 
   
  \begin{thm}[{{\cite[Theorem 7.2]{KR}}}] \label{thmKLR}
    \begin{enumerate}
      \item There exists a finite subset $\mathcal{GL}$ of $\mathcal{M}$ in bijection with $\Phi_{+}$ such that $\M$ is exactly the set 
      $$ \{ {\bf j}_1 \cdots {\bf j}_k \mid {\bf j}_1 , \ldots , {\bf j}_k \in \mathcal{GL} , {\bf j}_1 \geq \cdots \geq {\bf j}_k \} . $$
      \item For every $\mu \in \M$, there is a unique (up to isomorphism) finite-dimensional simple module $L(\mu)$ of highest word $\mu$. Moreover, write $\mu = {\bf j}_1 \cdots {\bf j}_k$; then $L(\mu)$ is given by
      $$ L(\mu) = hd \left(L({\bf j}_1) \circ \cdots \circ L({\bf j}_k) \right) . $$
      \item For $\mu$ of the form ${\bf j}^n$ with ${\bf j} \in \mathcal{GL}$, one has $L(\mu) = L({\bf j})^{\circ n}$.
    \end{enumerate}
 \end{thm}
The elements of $\mathcal{GL}$ are called \textit{good Lyndon words} (or dominant Lyndon words). Elements of $\M$ are called \textit{dominant words}. The simple modules corresponding to good Lyndon words are called \textit{cuspidal representations} (see also \cite{McN}). 
 For a given total order $<$ on $I$, the good Lyndon words can be constructed explicitly using the algorithm described in {{\cite[Section 4.3]{Leclerc}}}. 
 
  \begin{ex}
 Assume $\mathfrak{g}$ if of type $A_2$ and choose the order $1<2$. Then $\mathcal{GL} = \{ 1,12,2 \}$ ordered with respect to the lexicographic order in an obvious way, and the corresponding order on $\Phi_{+}$ is given by $\alpha_1 < \alpha_1 + \alpha_2 < \alpha_2$. 
   \end{ex}

 \begin{rk} \label{canofacto}
 For a dominant word $\mu \in \M$, the writing $\mu = {\bf i}_1 \cdots {\bf i}_k$ (with ${\bf i}_1 , \ldots , {\bf i}_k$ good Lyndon words ranged in the decreasing order) is known to be essentially unique (see \cite{KR} for a precise statement). It is called the \textit{canonical factorization} of $\mu$. 
 \end{rk}
  
\subsection{Monoidal categorification of quantum coordinate rings}
 \label{remindKK}

In their series of papers  \cite{KKK,KKKO3,KKKO}, Kang-Kashiwara-Kim-Oh constructed braiding operators for $R-gmod$ and showed that this category is a monoidal categorification (in the sense of \cite{HL}) of the cluster structure on $\Aqn$. They also proved similar statements for various subcategories of $R-gmod$, denoted $\Cw$. This section is devoted to fixing notations and recalling the main properties of these categories.
 Our exposition mainly follows {{\cite[Section 2.3]{KK}}}.

Assume $\mathfrak{g}$ is simply-laced. 
Let $W$ denote the Weyl group of $\mathfrak{g}$. For any $w \in W$, we denote by $N := l(w)$ the length of $w$. The quantum coordinate ring $\Aqnw$ is a subalgebra of $\Aqn$ defined in \cite{GLS}. It is shown to have a (quantum) cluster algebra structure. Kang-Kashiwara-Kim-Oh \cite{KKKO} showed that $\Aqnw$ admits a monoidal categorification by a subcategory $\Cw$ of $R-gmod$ ({{\cite[Theorem 11.2.3]{KKKO}}}). The category $\Cw$ is stable under taking subquotients, extensions, and monoidal products.  Following \cite{KK}, we set
$$ \Phi_{+}^w := \Phi_{+} \cap w \Phi_{-} $$
where $\Phi_{+}$ (resp. $\Phi_{-}$) stands for the set of positive (resp. negative) roots of $\mathfrak{g}$.  The set $\Phi_{+}^w$ has cardinality $N$ and we write $\Phi_{+}^w = \{\beta_k , 1 \leq k \leq N \}$. 

\begin{rk} 
The set of positive roots $\Phi_{+}^w$ does not depend on the choice of a reduced expression for $w$. Moreover, $\Phi_{+}^w \neq \Phi_{+}^{w'}$ if $w \neq w'$. 
\end{rk}

There is a natural bijection between (total) convex orderings on $\Phi_{+}^{w}$ and reduced expressions of $w$. More precisely, if ${\bf w} := (i_1, \ldots , i_N)$ is a reduced expression of $w$, then there is a natural corresponding convex order $\leq$ on $\Phi_{+}^w$ given by
$$ \beta_{1} < \cdots < \beta_{N} $$
where $\beta_k := s_{i_1} \cdots s_{i_{k-1}}(\alpha_{i_{k}})$ for every $1 \leq k \leq N$. We refer to \cite{Leclerc} and references therein for more details. 
Here we will be interested in a particular subfamily of convex orderings on $\Phi_{+}^{w}$ called Lyndon orderings. They are defined as the restrictions to $\Phi_{+}^{w}$ of orderings on $\Phi_{+}$ arising from the bijection of Theorem~\ref{thmKLR}(1), for each choice of a total order $<$ on $I$ (see also Section~\ref{sectionmainresult} below). 
Let ${\bf i}_{1} < \cdots < {\bf i}_{N}$ denote the good Lyndon words corresponding respectively to $\beta_{1}, \ldots , \beta_{N}$ via this bijection. It is known (see {{\cite[Section 2.3]{KK}}}) that the simple objects in $\Cw$ are exactly the $L(\mu)$ for $\mu$  dominant word of the form $({\bf i}_{N})^{c_{N}} \cdots ({\bf i}_{1})^{c_{1}} , c_{1}, \ldots , c_{N} \in \mathbb{N}$. 
 
\smallskip

Geiss-Leclerc-Schr\"oer \cite{GLS} constructed an initial seed $\mathcal{S}^{\bf w}$ in $\Aqnw$ corresponding to the chosen reduced expression ${\bf w}$ of $w$.  The index set $J$ of the cluster variables of this seed is $J=\{1, \ldots , N \}$; it splits into the disjoint union $J = J_{ex} \sqcup J_{fr}$ where $J_{ex}$ (resp. $J_{fr}$) denotes the index set of the unfrozen variables (resp. the frozen variables). One has 
$$ J_{fr} = \{k \in J , k_{+} = N+1 \} \qquad \text{and} \qquad J_{ex} = J \setminus J_{fr} $$
where $k_{+} := \min \left( \{k | k<s \leq N , i_k=i_s \} \sqcup \{N+1\} \right)$. 

\bigskip

	\subsection{Dominance order and generalized parameters}
	 \label{remindCasbi}
	 
	 We consider a Lie algebra $\mathfrak{g}$ of type $A_n, D_n$ or $E_n$ and we fix and index set  $I$ of simple roots. Recall that the category $R-gmod$ associated to $\mathfrak{g}$ is a monoidal category, whose product is denoted $\circ$. We let $<$ denote an arbitrary ordering on $I$. The set $\M$ of dominant words is totally ordered for the induced lexicographic order (still denoted $<$). It  is naturally endowed with a structure of abelian monoid, whose law is denoted $\odot$. By definition, for any $\mu , \nu \in \M$, $\mu \odot \nu$  is defined as the greatest of the dominant words corresponding to the Jordan-H\"older components of the product  $L(\mu) \circ L(\nu)$. The monoid $(\M, \odot)$ is commutative (as $K_0(\C)$ categorifies a cluster algebra, which is commutative); hence it is naturally embedded into its Grothendieck group $\G$ whose law is again denoted by $\odot$. This group is abelian and inherits a total lexicographic order that extends the one on $\M$ (see {{\cite[Definition 4.4, Proposition 4.5]{Casbi}}}). Let $r$ be the number of positive roots and let us write ${\bf j}_1 < \cdots < {\bf j}_r$ the elements of $\mathcal{GL}$ ordered with respect to $<$. Consider the following map:
	 $$ \begin{array}{cccc}
    \varphi : & \M & \longrightarrow & \mathbb{N}^{r} \\
     {} &  \mu & \longmapsto &  ~^t(c_1, \ldots , c_r)
   \end{array} 
   $$
   if $({\bf j}_r)^{c_r} \cdots ({\bf j}_{1})^{c_{1}}$ is the canonical factorization of $\mu$ (the integers $c_i$ being possibly zero). The following was proved for a type $A_n$ underlying Lie algebra $\mathfrak{g}$ in \cite{Casbi}:
   
   \begin{prop} \label{monoid}
  The map $\varphi$ is a monoid isomorphism from $(\M, \odot)$ to $(\mathbb{N}^{r},+)$. 
   \end{prop}
	 
	 We delay the proof to the next section.
   
	\smallskip
	
	 In {{\cite[Section 4.2]{Casbi}}}, we constructed a map 
   $$\tilde{\Psi} : K_0(R-gmod) \longrightarrow \M$$
    sending the class of any simple module in $R-gmod$ onto the corresponding dominant word in $\M$. This map satisfies 
$$ \tilde{\Psi}([L(\mu)][L(\nu)]) = \mu \odot \nu $$
for any $\mu, \nu \in \M$. In other words one has
  \begin{equation} \label{additivity}
   \tilde{\Psi}([M][N]) = \tilde{\Psi}([M]) \odot \tilde{\Psi}([N])
    \end{equation}
   for any simple objects $M,N$ in $R-gmod$. 
	 
	\bigskip
	
	 Now let $w \in W$ and $N := l(w)$.  Fix a reduced expression ${\bf w}$ of $w$. This is equivalent to the choice of order on $\Phi_{+}^{w}$ as recalled in Section~\ref{remindKK}. As in the previous section we write $\Phi_{+}^{w}=\{ \beta_1 < \cdots \beta_N \}$. By Theorem~\ref{thmKLR} there is a unique word ${\bf i}_k$ of $\mathcal{GL}$ with weight $\beta_k$ for every $1 \leq k \leq N$. Similarly the simple objects in $\Cw$ are in bijection with the set 
  $$\M_w := \{ ({\bf i}_N)^{c_N} \cdots ({\bf i})_1^{c_1} , c_1 , \ldots , c_N \geq 0 \} .$$
As before we let $\G_w$ denote the Grothendieck group of $\M_w$. 
   The rings $K_0(R-gmod)$ (resp. $K_0(\Cw)$) are domains and hence are embedded into their fraction fields. The map $\tilde{\Psi}$ can be extended to these fields by setting:
   $$ \tilde{\Psi} \left( \frac{x}{y} \right) := \tilde{\Psi}(x) \odot \tilde{\Psi}(y)^{\odot -1} . $$
 This map provides a way to study the cluster structure of $\Aqnw$ at the level of the monoid $\M_w$ (or the group $\G_w$). 
Let $\s$ be a seed in $\Aqnw$. Let $x_1, \ldots , x_N$ denote the cluster variables and $B=(b_{ij})_{i,j}$ the exchange matrix of $\s$. Following \cite{FZ4} we set
$$ \yjh := \prod_{ 1 \leq i \leq N} x_i^{b_{ij}} $$
and we define (see {{\cite[Definition 4.7]{Casbi}}}):
$$ \mjh := \tilde{\Psi} (\yjh) =  \bigodot_{1 \leq i \leq N} \mu_i^{\odot b_{ij} } \in G_w . $$
These elements are of particular interest from the perspective of monoidal categorification of cluster algebras as they can be used to define remarkable partial orderings as follows:

 \begin{deftn}[Dominance order , \cite{Qin} ] \label{defdom}
  Let $\s = ((x_1, \ldots , x_N) , B)$ be a seed in $\A$ and consider the elements $\yjh$ defined above. Then, given $\mathfrak{m}$ and $\mathfrak{m}'$ two Laurent monomials in the $x_i$, we write 
   $$ \mathfrak{m}  \preccurlyeq  \mathfrak{m}' \Leftrightarrow  \exists  \gamma_1, \ldots , \gamma_n \geq 0 ,  \mathfrak{m}' = \mathfrak{m} \times \prod_j {\yjh}^{\gamma_j}. $$
   \end{deftn}
   
   \begin{rk}
Note that we take opposite conventions compared with \cite{Qin}. 
 \end{rk}   
   
 In the framework of monoidal categorification of quantum coordinate rings using quiver Hecke algebras, Kashiwara-Kim \cite{KK} related this cluster-theoretic partial ordering on monomials to homogeneous degrees of renormalized $R$-matrices constructed in \cite{KKK}. 
In \cite{Casbi}, we used the $\mjh$ as natural analogs of $\yjh$ in terms of parameters for simple modules of quiver Hecke algebras. We introduced the notion of \textit{compatible seed} (see {{\cite[Definition 4.7]{Casbi}}}). 
The following statement is the main result of \cite{Casbi}.

 \begin{thm}[{{\cite[Theorem 6.2]{Casbi}}}]
 Take $\mathfrak{g} = \mathfrak{sl}_{n+1}$ and let $w=w_0$ be the longest element of the corresponding Weyl group. Consider the reduced expression 
 $$ {\bf w_0} := (1,2,1,3,2,1,\ldots , n,n-1, \ldots,1) . $$
 Then the seed $\mathcal{S}^{\bf w_0}$ is a compatible seed in $R-gmod$. 
 \end{thm}
  

 \section{Seeds associated with orderings on Lyndon words}
 \label{seedsection}

In this section we generalize several results obtained in \cite{Casbi}. Let $\mathfrak{g}$ be a  semisimple Lie algebra of finite type, $I$ a fixed index set of simple roots and $<$ a total ordering on $I$. We begin by proving Proposition~\ref{monoid} in the following cases: $\mathfrak{g}$ of arbitrary finite type  with the choice of the natural ordering of $I$ (see {{\cite[Section 8]{KR}}) and $\mathfrak{g}$ of classical type with any ordering. This generalizes {{\cite[Proposition 5.1]{Casbi}}}. Then we assume $\mathfrak{g}$ is simply-laced and for any $w$ in the Weyl group $W$ of $\mathfrak{g}$, we consider a reduced expression ${\bf w}_{<}$ of $w$ uniquely determined by $<$, together with the corresponding seed $\mathcal{S}^{{\bf w}_{<}}$  in $\Aqnw$ following \cite{GLS,KKKO}. We provide an explicit description in terms of dominant words of the simple modules corresponding to the cluster variables of $\mathcal{S}^{{\bf w}_{<}}$ (see Theorem~\ref{initpara}). This holds for arbitrary $<$ and $\mathfrak{g}$ of finite type. It generalizes {{\cite[Theorem 6.1]{Casbi}}} to any subcategory $\Cw$ (not only $R-gmod$) and any finite-type underlying Lie algebra $\mathfrak{g}$. The key tool for the proof is provided by {{\cite[Proposition 3.14]{KK}}}. We state several consequences, and in particular we prove {{\cite[Conjecture 4.10]{Casbi}}} in the cases where Proposition~\ref{monoid} hods.

 \subsection{The monoid structure on dominant words }

 This subsection is devoted to the proof of Proposition~\ref{monoid} for $\mathfrak{g}$ of arbitrary finite type. Let us outline the fact that for $\mathfrak{g}$ of exceptional type, we need to fix the natural order on $I$, as in {{\cite[Section 8.2]{KR}} (see Corollary~\ref{cor2} below). The  cuspidal representations in $R-gmod$ (parametrized by elements of $\mathcal{GL}$) are described in {{\cite[Sections 8.8-8.10]{KR}}. For arbitrary orderings, our proof works for $\mathfrak{g}$ of any classical type.

We let $<$ denote an arbitrary total ordering on $I$.
Recall that a Lyndon word is by definition  a word which is strictly smaller than any of its proper right factors. 

Let $\mu, \nu$ be two words of respective lengths $m,n$. We denote by $\mu \nu$ the concatenation of $\mu$ and $\nu$. 
By \textit{shuffle} of $\mu$ and $\nu$ we mean a word obtained by applying a permutation $\sigma$ to the letters of $\mu \nu$ such that the restrictions of $\sigma$ to $\{1,\ldots , m \}$ and $\{m+1,\ldots , m+n \}$ are increasing (see {{\cite[Section 4.2]{KR}}}). We denote this word by $\sigma \cdot (\mu \nu)$.

 We consider two dominant words $\mu , \nu \in \M$ that we write 
 $$ \mu = {\bf i}_1 \cdots {\bf i}_r \quad , \quad \nu = {\bf j}_1 \cdots {\bf j}_s $$
 where ${\bf i}_1 , \ldots ,  {\bf i}_r , {\bf j}_1 , \ldots , {\bf j}_s \in \mathcal{GL}$ with ${\bf i}_1 \geq \cdots \geq {\bf i}_r$ and ${\bf j}_1 \geq \cdots \geq {\bf j}_s$.  We show that 
\begin{equation} \label{candidate}
\mu \odot \nu =  {\bf l}_1 \cdots {\bf l}_{r+s} 
\end{equation}
where ${\bf l}_1 \geq \cdots \geq {\bf l}_{r+s}$ are the elements of the set $\{ {\bf i}_1 , \ldots , {\bf i}_r , {\bf j}_1 , \ldots , {\bf j}_s \}$ ranged in the decreasing order. 
Equivalently, we show that the right hand side of Equation~\eqref{candidate} is the biggest shuffle of $\mu$ and $\nu$ (for the lexicographic order). As it is obviously a shuffle of $\mu$ and $\nu$, it only remains to show that it is bigger than  any other shuffle of $\mu$ and $\nu$.

We use an induction on $r+s$, i.e. we assume that ${\bf i}_2 \cdots {\bf i}_r \odot {\bf j}_1 \cdots {\bf j}_s$ is the concatenation of ${\bf i}_1 , \ldots , {\bf i}_r , {\bf j}_2 ,  \ldots , {\bf j}_s $ and that ${\bf i}_1 \cdots {\bf i}_r \odot {\bf j}_2 \cdots {\bf j}_s$ is the concatenation of  $ {\bf i}_2 , \ldots , {\bf i}_r , {\bf j}_1 , \ldots , {\bf j}_s$ ranged in the decreasing order. 

  Let $\sigma$ be a shuffle permutation of $\mu$ and $\nu$. We show that $\sigma \cdot (\mu  \nu) \leq {\bf l}_1 \cdots {\bf l}_{r+s}$. We assume ${\bf i}_1 \geq {\bf j}_1$ the other case being analogous. Note that this is equivalent to ${\bf l}_1 = {\bf i}_1$. 
 Let us write 
 $$ {\bf i}_1 = (a_1, \ldots , a_p) \quad , \quad {\bf j}_1 = (b_1 , \ldots , b_q). $$

   \begin{lem}  \label{lem2}
   If $a_1 > b_1$ then $\sigma$ shuffles ${\bf j}_1$ to the right of ${\bf i}_1$. Equivalently $\sigma(1)=1 , \ldots , \sigma(p)=p$.  
   \end{lem}
 
\begin{proof}
As the restrictions of $\sigma$ to $\{1,\ldots , m \}$ and $\{m+1,\ldots , m+n \}$ are increasing, the first letter of $\sigma \cdot (\mu \nu)$ is either $a_1$ or $b_1$. As $b_1 < a_1$, $\sigma \cdot (\mu \nu)$ cannot begin with $b_1$ as it would then be strictly smaller than the right hand side of Equation~\eqref{candidate}. Hence the first letter of $\sigma \cdot (\mu \nu)$ is $a_1$.  Then the second letter is either $a_2$ or $b_1$. As ${\bf i}_1$ is Lyndon,  $a_2 \geq a_1$ and thus $a_2 > b_1$. Hence as before the second letter of $\sigma \cdot (\mu \nu)$ has to be $a_2$. We conclude by a straightforward induction that $\sigma \cdot (\mu \nu)$ begins with ${\bf i}_1$ which proves the Lemma. 
\end{proof}
 
From now on we always assume ${\bf i}_1 \geq {\bf j}_1$ and $a_1 = b_1$. 
 
 \begin{lem} \label{lem3}
 Assume there is only one occurrence of $a_1$ in ${\bf i}_1$ i.e. one has $a_k > a_1$ for every $k \geq 2$. Then $\sigma$ shuffles ${\bf j}_1$ to the right of ${\bf i}_1$ if ${\bf i}_1 > {\bf j}_1$ and either to the left either to the right of ${\bf i}_1$ if ${\bf i}_1 = {\bf j}_1$. 
 \end{lem}
 
 \begin{proof}
  Consider $k$ maximal in $\{1 , \ldots p\}$ such that $a_1=b_1 , \ldots , a_k=b_k$. 
 Assume $k<p$; then $a_{k+1} > b_{k+1}$ and ${\bf i}_1 > {\bf j}_1$.  If $\sigma$ shuffles $b_1$ to the first letter of $\sigma \cdot (\mu \nu)$  then the second letter is either $a_1$ or $b_2$. But $(b_1 , a_1 \cdots) = (a_1 , a_1 \cdots) < (a_1 , a_2 \cdots)$ by assumption hence $\sigma$ shuffles $b_2$ to the second letter of $\sigma \cdot (\mu \nu)$. Similarly we get that $\sigma$ shuffles $b_1, \ldots , b_k$ to the first $k$ letters of $\sigma \cdot (\mu \nu)$. Then the next letter is either $a_1$ or $b_{k+1}$. Both of these letters are strictly smaller than $a_{k+1}$ which contradicts the fact that $\sigma \cdot (\mu \nu)$ is greater than the right hand side of Equation~\eqref{candidate}. Hence we proved that the first letter of $\sigma \cdot (\mu \nu)$ is $a_1$. 
 
 Then the second letter is either $a_2$ or $b_1$ but $b_1 = a_1 < a_2$ by assumption hence it has to be $a_2$. Another induction shows that $\sigma \cdot (\mu \nu)$ begins with ${\bf i}_1$. 
 
  If $k=p$ then ${\bf i}_1 = {\bf j}_1$. The same proof shows that the first letter of $\sigma \cdot (\mu \nu)$ is either $a_1$ and in this case $\sigma \cdot (\mu \nu)$ begins with ${\bf i}_1$, or $b_1$ and in this case $\sigma \cdot (\mu \nu)$ begins with ${\bf j}_1$. 
  
  \end{proof}
  
The two previous lemmas were essentially proved in type $A_n$ in {{\cite[Section 5.1]{Casbi}}}. In order to deal with the remaining types, we consider a slightly more general version of Lemma~\ref{lem3}. Recall that we assume ${\bf i}_1 \geq {\bf j}_1$ and $a_1=b_1$.  

\begin{lem}  \label{lem4}
  Assume there are exactly two occurrences of $a_1$ in ${\bf i}_1$ and exactly one in ${\bf j}_1$.  
  Then ${\bf j}_1$ is shuffled to the right of ${\bf i}_1$. 
  \end{lem}
  
  \begin{proof}
  We write  ${\bf i}_1 = a {\bf k} a {\bf l} $, where ${\bf k}$ and ${\bf l}$ are words whose letters are all strictly greater than $a$, and ${\bf j}_1 = a {\bf m}$.
  First note that the assumptions imply that the word ${\bf k}$ is not empty. 
  As ${\bf i}_1 \geq {\bf j}_1$, one has ${\bf k} \geq {\bf m}$ and thus $a {\bf k} \geq {\bf j}_1$. Moreover ${\bf k} < {\bf l}$ as ${\bf i}_1$ is Lyndon. Hence one has ${\bf j}_1 = a {\bf m} \leq a {\bf k} < a {\bf l}$. 
  
   Lemma~\ref{lem3} implies that ${\bf j}_1$ is shuffled either to the left of ${\bf i}_1$, or between $a{\bf k}$ and $a{\bf l}$, or to the right of ${\bf i}_1$. The first possibility is possible only if ${\bf m}={\bf k}$. But in this case one can apply Lemma~\ref{lem3} in the same way with ${\bf j}_2$. Hence the result of the shuffle would begin either with ${\bf j}_1 {\bf j}_2$ or ${\bf j}_1 a {\bf k}$. Both are strictly smaller than $a {\bf k}a{\bf l} = {\bf i}_1$. 
   
    Now Lemma~\ref{lem3} applied to $a{\bf l}$ and ${\bf j}_1$ implies that ${\bf j}_1$ has to be shuffled to the right of ${\bf i}_1$ which finishes the proof. 
 \end{proof}
  
  \begin{lem} \label{lem5}
   Assume  there are exactly two occurrences of $a_1$ in ${\bf i}_1$ as well as in ${\bf j}_1$.   
  Then ${\bf j}_1$ is shuffled to the right of ${\bf i}_1$ if ${\bf i}_1 > { \bf j}_1$, and either to the left either to the right of ${\bf i}_1$ if ${\bf i}_1 = {\bf j}_1$.
  \end{lem}

\begin{proof}
We write ${\bf i}_1 = a {\bf k} a {\bf l}$ and ${\bf j}_1 = a {\bf m} a {\bf n}$ where ${\bf k}, {\bf l}, {\bf m}, {\bf n}$ are words whose letters are all strictly greater than $a$. 
As in the previous Lemma, one has ${\bf k} < {\bf l}$ and ${\bf m} < {\bf n}$. Moreover one has ${\bf k} \geq {\bf m}$, and in case of equality one has ${\bf n} \leq {\bf l}$. 

 If ${\bf i}_1 > {\bf j}_1$ then either ${\bf k} <{\bf m}$ or $ {\bf k} = {\bf m} $ and ${\bf n} < {\bf l}$. The same proof as for Lemma~\ref{lem4} show that ${\bf j}_1$ is shuffled to the right of ${\bf i}_1$. 
 The case  ${\bf i}_1 = {\bf j}_1$ also follows from the same arguments. 
 \end{proof}

  \begin{cor} \label{cor1}
  Assume ${\bf i}_1$ and ${\bf j}_1$ contain at most two occurrences of their first letters. 
 Then $\sigma \cdot (\mu \nu)$ is either the concatenation of ${\bf i}_1$ with ${\bf i}_2 \cdots {\bf i}_r \odot  {\bf j}_1 \cdots {\bf j}_s$ or the concatenation of ${\bf j}_1$ with ${\bf i}_1 \cdots {\bf i}_r \odot  {\bf j}_2 \cdots {\bf j}_s$. 
  \end{cor}
  
  \begin{proof}
  If ${\bf i}_1 > {\bf j}_1$ then ${\bf l}_1 = {\bf i}_1$ and the previous lemmas show that $\sigma \cdot (\mu \nu)$ is the concatenation of ${\bf i}_1$ with a shuffle of ${\bf i}_2 \cdots {\bf i}_r$ and ${\bf j}_1 \cdots {\bf j}_s$. By the induction assumption, any such shuffle is smaller than ${\bf l}_2 \cdots {\bf l}_{r+s}$. Hence $\sigma \cdot (\mu \nu) \leq {\bf i}_1 {\bf l}_2 \cdots {\bf l}_{r+s} = {\bf l}_1 \cdots {\bf l}_{r+s}$. The case ${\bf i}_1 < {\bf j}_1$ is analogous. 
  
   If ${\bf i}_1 = {\bf j}_1$, then ${\bf l}_1 = {\bf l}_2 = {\bf i}_1$. By the previous lemmas, $\sigma \cdot (\mu \nu)$ is either the concatenation of ${\bf l}_1$ with a shuffle of ${\bf i}_2 \cdots {\bf i}_r$ and ${\bf j}_1 \cdots {\bf j}_s$ or the concatenation of ${\bf l}_1$ with a shuffle of ${\bf i}_1 \cdots {\bf i}_r$ and ${\bf j}_2 \cdots {\bf j}_s$. In both cases the conclusion is the same. 
  \end{proof}
 
\begin{cor} \label{cor2}
Assume $\mathfrak{g}$ is of arbitrary finite type and $<$ is the natural ordering on $I$ or $\mathfrak{g}$ is of classical type and $<$ is an arbitrary ordering on $I$. Then Equation~\eqref{candidate} holds.
 \end{cor}
 
 \begin{proof}
 For $\mathfrak{g}$ of types $A,B,C,D$, the positive roots contain at most to occurences of any simple root. A fortiori the elements of $\mathcal{GL}$ contain at most two occurrences of their first letter and this holds for any ordering $<$ on $I$. When $\mathfrak{g}$ is of exceptional type and $<$ is the natural ordering, the elements of $\mathcal{GL}$ are described in {{\cite[Section 8.2]{KR}}} and one can see that they always contain at most two occurrences of their first letter. Hence in these cases, Corollary~\ref{cor1} implies that the right hand side of Equation~\eqref{candidate} is the greatest shuffle of $\mu$ and $\nu$, proving Proposition~\ref{monoid}.
 
 \end{proof}

 \subsection{A compatible seed for $\Cw$}
  \label{sectionmainresult}

 In this subsection we consider $\mathfrak{g}$ of type $A_n, D_n$ or $E_n$  and we prove that for any $w \in W$, there exists a reduced expression ${\bf w}$ of $w$ such that the seed $\mathcal{S}^{\bf w}$ is compatible in the sense of {{\cite[Definition 4.7]{Casbi}}}.
 
 We fix an arbitrary total order $<$ on $I$. We again denote by $<$ the induced lexicographic order on the set $\mathcal{M}$ as well as its restriction to $\mathcal{GL}$ (see Section~\ref{remindKLR}). Via the bijection between $\mathcal{GL}$ and $\Phi_{+}$, one can view $<$ as an ordering on $\Phi_{+}$. For any $w \in W$, we consider the restriction of $<$ to $\Phi_{+}^{w}$. As there is a bijection between convex orderings on $\Phi_{+}^{w}$ and reduced expressions of $w$ (see Section~\ref{remindKK}), we consider the unique reduced expression ${\bf w}_{<}$ of $w$ corresponding to $<$.    
 We fix once for all the order $<$ and we write ${\bf w}$ instead of ${\bf w}_{<}$ if there is no ambiguity.

 First we introduce a notation that will be useful in the following. For any $1 \leq k \leq N$, we set 
 $$ J_{k} := \{j \leq k | i_j=i_k \} $$
  and we write $J_k = \{j_0=k > j_1 > \cdots > j_{r_k} \}$. In other words, $j_{0} = k , j_{1} = k_{-} , j_{2}=  (k_{-})_{-}, \ldots $ with the notations of {{\cite[Section 9.4]{GLS}}}. The integer $r_{k}$ corresponds to the position of the first occurrence of the letter $i_k$ in the word $(i_1, \ldots , i_N)$.
  
  \smallskip

The following statement is the main result of this section. It gives a description of the simple modules in $\Cw$ corresponding to the cluster variables of the seed  $\mathcal{S}^{\bf w}$.

 \begin{thm}  \label{initpara}
Let $(x_1, \ldots , x_N)$ denote the cluster variables of the seed $\mathcal{S}^{\bf w}$ and let $\mu_k$ denote the dominant word such that $x_k=[L(\mu_k)]$ for every $k \in J$. Then 
 $$ \mu_k = {\bf i}_{j_{0}} {\bf i}_{j_{1}} \cdots {\bf i}_{j_{r_k}} . $$
 \end{thm}
 
 \begin{proof}
 We write the canonical factorization of $\mu_k$ as 
 $$ \mu_k = ({\bf i}_N)^{c_{N,k}} \cdots ({\bf i}_1)^{c_{1,k}}  $$
 with  $c_{N,k}, \ldots , c_{1,k} \in \mathbb{N}$. We also set $c_{(N+1),N} := 0$.   By {{\cite[Proposition 3.14]{KK}}}, the t-uple of integers $(c_{1,k}-c_{1_{+},k} , \ldots , c_{N,k}-c_{N_{+},k})$ is the image of $[L(\mu_k)]$ under the map ${\bf g}_{\mathcal{S}_{0}^{\bf w}}^{R}$ defined in \cite{KK} (see {{\cite[Definition 3.8]{KK}}}). It is clear from this definition that the isomorphism class of the simple module $L(\mu_k)$ is mapped onto the $k$th vector $e_k$ of the standard basis of $\zn$. Thus one has 
 $$ e_k = (c_{1,k}-c_{1_{+},k} , \ldots , c_{N,k}-c_{N_{+},k}) .$$
    One has $c_{j,k} - c_{j_{+},k} = 0 $ for any $j \neq k$. In particular, one has 
    $$c_{k_{+},k} = c_{(k_{+})_{+},k} = \cdots = c_{(N+1),k} := 0  \quad \text{and} \quad  c_{r_{k},k} =  \cdots = c_{k_{-},k} = c_{k,k} . $$ 
    Moreover, $c_{k,k}-c_{k_{+},k} = 1$ and hence $c_{k,k}=1$. Finally one has 
    $$ c_{r_{k},k} =  \cdots = c_{k_{-},k} = c_{k,k} = 1 \quad \text{and} \quad c_{k_{+},k} = c_{(k_{+})_{+},k} = \cdots = 0  . $$
  If $j$ is any position such that the letter $i_j$ is different from $i_k$ then $k$ does not appear in the sequence $(r_{j}, \ldots , j_{-}, j , j_{+} , \ldots , N+1)$ and hence $c_{r_{j},k} =  \cdots = c_{j_{-},k} = c_{j,k} =  c_{j_{+},k} = \cdots = 0$. One concludes:
  $$ \mu_k = {\bf i}_k {\bf i}_{k_{-}} \cdots {\bf i}_{r_{k}} . $$

 \end{proof}

Let us point out a couple of consequences which will be useful later. 

\begin{cor} \label{ordrepara}
 Let $k \in J$. For any integer-valued t-uple $(c_{j})_{j \in J, j > k}$, one has
 $$ \mu_k > \bigodot_{j < k} \mu_j^{\odot c_j} $$ 
 in the group $\G$. 
\end{cor}

\begin{proof}
 By Theorem~\ref{initpara}, the highest good Lyndon word in the canonical factorization of $\mu_k$ (resp. $\mu_j , j>k$) is ${\bf i}_k$ (resp. ${\bf i}_j , j>k$). Hence by definition of the lexicographic order on $\M$ one has 
 $$ \mu_k \odot \bigodot_{j < k , c_j <0} \mu_j^{\odot -c_j} > \bigodot_{j < k , c_j >0} \mu_j^{\odot c_j}$$
 which implies 
 $$ \mu_k > \bigodot_{j < k} \mu_j^{\odot c_j} $$ 
 in the Grothendieck group $\G$ of $\M$. 
\end{proof}

\begin{cor} \label{invert}
 Let $\s = ((x_1, \ldots , x_N),B)$ be any seed in $\Aqnw$ and let $\mathcal{M}_{\s}$ denote the matrix of the vectors $\varphi(\tilde{\Psi}(x_1)), \ldots , \varphi(\tilde{\Psi}(x_N))$ in the standard basis of $\zn$. Then $\mathcal{M}_{\s} \in GL_N(\mathbb{Z})$.
 \end{cor}
 
 \begin{proof}
 First consider the seed $\mathcal{S}^{\bf w}$. By Theorem~\ref{initpara}, there is a bijection between the cluster variables of $\mathcal{S}^{\bf w}$ and good Lyndon words in $\M_w$: indeed, for any $1 \leq j \leq N$, there is a unique cluster variable in $\mathcal{S}^{\bf w}$ whose corresponding dominant word has a canonical factorization beginning with ${\bf i}_j$. Hence choosing a good permutation of the standard basis of $\zn$, the matrix $\mathcal{M}_{\mathcal{S}_{0}^{\bf w}}$ is (lower) unitriangular. In particular $\mathcal{M}_{\mathcal{S}_{0}^{\bf w}}$ is  invertible with determinant equal to $1$.
 
  Now consider a mutation in any direction $k \in J_{ex}$. Set $\Psi := \varphi \circ \tilde{\Psi}$. The vector $\Psi(x'_k)$ is either equal to $-\Psi(x_k) + \sum_{b_{ik}>0} b_{ik} \Psi(x_i)$ or to $-\Psi(x_k) + \sum_{b_{ik}<0} (-b_{ik}) \Psi(x_i)$. In the first case, one has
  \begin{align*}
  \det(\Psi(x'_1) , \ldots , \Psi(x'_N))  
   &= \det( \Psi(x_1), \ldots , \Psi(x_{k-1}), \Psi(x'_k) , \Psi(x_{k+1}) , \ldots , \Psi(x_N)) \\
   &=- \det( \Psi(x_1), \ldots , \Psi(x_{k-1}), \Psi(x_k) ,\Psi(x_{k+1}) , \ldots , \Psi(x_N))  \\
     + & \sum_{b_{ik}>0}  \det( \Psi(x_1), \ldots , \Psi(x_{k-1}), \Psi(x_i) , \Psi(x_{k+1}) , \ldots , \Psi(x_N)) \\
     &= - \det(\Psi(x_1) , \ldots , \Psi(x_N)).
 \end{align*}
 The other case is analogous. 
 In particular the matrix obtained after mutation is still invertible and has determinant either $1$ or $-1$. By induction, we conclude that $\mathcal{M}_{\s} \in G_N(\mathbb{Z})$ for any seed $\s$. 
 \end{proof}
 
  We end this section by proving that {{\cite[Conjecture 4.10]{Casbi}}} holds in $\Cw$ for every $w \in W$ when the underlying Lie algebra $\mathfrak{g}$ is of finite type. The proof is similar to the proof of {{\cite[Theorem 6.2]{Casbi}}} in the case of $\C_{w_0} = R-gmod$ in type $A_n$. 
  
  \begin{cor} \label{corcompat}
 Fix any total order $<$ on $I$. For any $w \in W$, the seed $\mathcal{S}
 ^{{\bf w}_{<}}$ is compatible in the sense of {{\cite[Definition 4.7]{Casbi}}}. 
\end{cor}

\begin{proof}
As in Theorem~\ref{initpara}, for every $1 \leq k \leq N$ we let $\mu_k$ denote the dominant word corresponding to the $k$th cluster variable of $\mathcal{S}^{{\bf w}_{<}}$. With the same notations as in Section~\ref{remindCasbi}, we consider the variables $\yjh , j \in J_{fr}$. It follows from the construction of $\mathcal{S}^{\bf w}$ (\cite{GLS,KKKO}) that for every $j \in J_{fr}$,
$$ \yjh = x_{j_{+}} x_{j_{-}}^{-1} \prod_{j<k<j_{+}<k_{+}} x_k^{-|a_{kj}|} \prod_{k<j<k_{+}<j_{+}} x_k^{|a_{kj}|} . $$
Hence 
$$ \mjh = \mu_{j_{+}} \odot \left( \mu_{j_{-}}^{\odot -1} \odot \bigodot_{k<j<k_{+}<j_{+}} \mu_k^{\odot |a_{kj}|} \odot \bigodot_{j<k<j_{+}<k_{+}} \mu_k^{\odot -|a_{kj}|} \right) . $$
The expression between brackets only involves words $\mu_k$ such that $k<j_{+}$. Hence by Corollary~\ref{ordrepara}, 
$$ \mu_{j_{+}}  >  \left( \mu_{j_{-}}^{\odot -1} \odot \bigodot_{k<j<k_{+}<j_{+}} \mu_k^{\odot |a_{kj}|} \odot \bigodot_{j<k<j_{+}<k_{+}} \mu_k^{\odot -|a_{kj}|} \right) . $$
Thus one has $\mjh \odot \mu > \mu$ for every $\mu \in \M_w$ and this holds for every $j \in J_{fr}$. This implies that the seed $\mathcal{S}^{\bf w}$ is compatible (see {{\cite[Remark 4.8]{Casbi}}}). 
\end{proof}

 \begin{rk}
 By {{\cite[Corollary 4.12]{Casbi}}}, the seed $\mathcal{S}^{\bf w}$ being compatible implies certain relationships between dominant words and $g$-vectors. This relationship is provided by {{\cite[Proposition 3.14]{KK}}} for any $w \in W$. It takes the form expected in {{\cite[Section 7.1]{Casbi}}} in the case of $R-gmod$ in type $A_n$ for the natural ordering. 
 \end{rk}
   
   \begin{ex} \label{example_sl3}
   Consider $\mathfrak{g}$ of type $A_2$, $I=\{1,2\}$, $w =w_0=s_1s_2s_1=s_2s_1s_2$. Consider the natural ordering $1<2$. Then $\Phi_{+}=\{ \alpha_1 < \alpha_1 + \alpha_2 < \alpha_2 \}$. The corresponding reduced expression of $w_0$ is $(1,2,1)$. It is known from {{\cite[Section 8.4]{KR}}} that $\mathcal{GL} = \{ (1) < (12) < (2) \}$. By Theorem~\ref{initpara}, the simple modules corresponding to the cluster variables of the seed $\mathcal{S}^{(1,2,1)}$ (together with its quiver) are given by 
   \smallskip
   $$ \xymatrix{ L(1) \ar[r] & L(12) \ar[r] & L(21) \ar@/_1pc/[ll] } . $$
 The matrix $\mathcal{M}_{\mathcal{S}^{(1,2,1)}}$ is 
  $$ \begin{pmatrix}
   1 & 0 & 1 \\
   0 & 1 & 0 \\
   0 & 0 & 1
  \end{pmatrix} .  $$
  \end{ex} 
   
   \section{Newton-Okounkov bodies for $\Cw$}
    \label{NObodyKLR}

It follows from Theorem~\ref{initpara} that for every choice of order $<$ on $I$ and every $w \in W$, there is a natural total ordering on the set of cluster variables of $\mathcal{S}^{{\bf w}_{<}}$ (and hence of every seed in $\Aqnw$). The corresponding lexicographic order on cluster monomials yields a monomial valuation for every seed  as in {{\cite[Section 7]{RW}}}. Thus it is natural to construct Newton-Okounkov bodies in this framework. It will turn out that in our setting the valuation will be naturally provided by parametrizations of simple objects in $R-gmod$ (or $\Cw$) and hence  entirely determined by $<$. In particular it does not depend on the choice of a seed. 
 Throughout the following sections, we consider a semisimple Lie algebra $\mathfrak{g}$ of arbitrary finite type. As in Section~\ref{sectionmainresult}, we fix an index set $I$ of the simple roots of $\mathfrak{g}$  and a total order $<$ on $I$. We also fix an element $w$ in the Weyl group $W$ of $\mathfrak{g}$ as well as the reduced expression ${\bf w}_{<} = (i_1, \ldots , i_N)$ corresponding to the restriction of $<$ to $\Phi_{+}^{w}$. In all Sections~\ref{NObodyKLR} and~\ref{sectionds}, we will use the following notations: $\C := \Cw, \A := K_0(\C)  \simeq \Aqnw, \M := \M_w, \G=\G_w$. 

 In order to construct Newton-Okounkov bodies for subalgebras of $\A$, we begin by  constructing a valuation with value in $\zn$ endowed with some total ordering, as well as a $\mathbb{N}$-graduation on $\A$.

   \subsection{Newton-Okounkov bodies}
    \label{NObodies}
    
    In this section we briefly review the general construction of Newton-Okounkov bodies, introduced by Kaveh-Khovanskii \cite{KKh} and independently by Lazarsfeld-Mustata \cite{LM}. It generalizes a construction of Okounkov \cite{Oko}. We refer to \cite{Bouck,BossPhD} for beautiful surveys on Newton-Okounkov bodies.
 
  Let $\A$ be a $\mathbb{N}$-graded commutative algebra over a base field ${\bf k}$. Let $\A_n$ denote the degree $n$ homogeneous subspace of $\A$ for any $n \in \mathbb{N}$. Thus one has
   $$ \A = \bigoplus_n {\A}_n   \quad , \quad {\A}_n {\A}_m \subset {\A}_{n+m} \quad  , \quad {\A}_0 = {\bf k}. $$
  We assume $\A_n$ to be a finite-dimensional ${\bf k}$-vector space for every $n \in \mathbb{N}$. We also assume that $\A$ is a domain and that the fraction field of $\A$ is of finite type over ${\bf k}$. 
  
  \begin{deftn}[Valuation] \label{defvalu}
   A valuation on $\A$ is a map $\Psi: \A \longrightarrow \zn$ (for some $N \geq 1$) satisfying the following properties:
    \begin{enumerate}
     \item $\forall f,g \in \A, \Psi(fg) = \Psi(f) + \Psi(g).$
     \item $\forall t \in {\bf k}^{*}, \forall f \in \A, \Psi(tf) = \Psi(f).$
     \item $\forall f,g \in \A, \Psi(f+g) \leq \max(\Psi(f),\Psi(g)).$
    \end{enumerate}
 \end{deftn}
 
 To any $\mathbb{N}$-graded subalgebra $\B$ of $\A$, one can associate a closed convex set $\Delta(\B)$ called Newton-Okounkov body of $\B$. 
 
  \begin{deftn}[Newton-Okounkov bodies]
   Let $\B$ be any graded subalgebra of $\A$. Decompose it as 
$$ \B = \bigoplus_n \B_n . $$   
    The Newton-Okounkov body associated to $\B$ is defined as 
   $$ \Delta(\B) := \overline{ \text{ConvexHull} \left( \bigcup_n  \frac{1}{n} \Psi(\B_n \setminus \{0\}) \right)}. $$
   \end{deftn}
   
  Recall that the vector spaces $\B_n , n \in \mathbb{N}$ are finite-dimensional. Hence it follows from {{\cite[Proposition 2.3]{KKh}}} that the sets $\Psi(\B_n \setminus \{0\})$ are finite. 
 In order to have these bodies satisfying nice properties, one needs to make a technical assumption on the valuation $\Psi$, namely that it is of maximal rational rank. This means that the rank of the value group of $\Psi$ has to be equal to the transcendence degree of $K := Frac(\A)$. We refer to {{\cite[Section 2.4]{Bouck}}} for more details and precise statements about maximal rational rank valuations. 
  The crucial observation is that under this assumption, one has 
  $$ \dim_{{\bf k}} \B_n =\sharp \Psi(\B_n \setminus \{0\}) .$$
  In other words if the valuation $\Psi$ is of maximal rational rank, then it has \textit{one-dimensional leaves} in the terminology of \cite{KKh} (see {{\cite[Proposition 2.23]{Bouck}}}). 
 
 \subsection{Total order on $\zn$} \label{totord}
 
As in Section~\ref{remindKK}, we write
 $$\Phi_{+}^w = \{ \beta_1 < \cdots < \beta_N \}$$ 
  and we let ${\bf i}_1 , \ldots , {\bf i}_N$ denote the corresponding good Lyndon words. Let ${\bf e}_k$ be the $k$th vector of the standard basis of $\zn$ for every $1 \leq k \leq N$. Recall from Section~\ref{remindCasbi} that the isomorphism $\varphi$ sends the good Lyndon word ${\bf i}_k$  onto ${\bf e}_k$. Equivalently, one has
 \begin{equation} \label{PsiLyndon}
  \forall 1 \leq k \leq N, \varphi \left( \tilde{\Psi} \left( [L({\bf i}_k)] \right) \right) = {\bf e}_k .
  \end{equation}
 Using the isomorphism $\varphi$, one can push forward the lexicographic order on $\M$ (resp. $\G$) onto the (reversed) lexicographic order on $\mathbb{N}^N$ (resp. $\zn$) given by:
 $$ (a_1, \ldots , a_N) < (b_1, \ldots , b_N) \Leftrightarrow  \exists k \geq 1, a_N=b_N, \ldots , a_{k+1}=b_{k+1} , a_k < b_k . $$
 
 \bigskip

  \subsection{The valuation} \label{valudef}
  
  We let $\Psi : Frac(K_0(R-gmod)) \longrightarrow \mathbb{Z}^{\Phi_{+}}$ denote the composition of $\tilde{\Psi}$ with $\varphi$:
 $$ \Psi : \xymatrixcolsep{4pc}\xymatrix{ Frac(K_0(R-gmod)) \ar[r]^{\qquad \qquad \tilde{\Psi}} & \G \ar[r]^{\varphi} & \mathbb{Z}^{\Phi_{+}} }. $$
  We again denote by $\Psi$ the restriction to $Frac(\A)$:
  $$ \Psi : \xymatrixcolsep{4pc}\xymatrix{ Frac(\A)  \ar[r]^{ \qquad \tilde{\Psi}} & \G \ar[r]^{\varphi} & \zn } . $$
  Recall that $N$ denotes the length of $w$. 
 
  \begin{lem} \label{valu}
   The map $\Psi$ is a valuation on $\A$ with value group $\zn$. 
   \end{lem} 

 \begin{proof}
Let $x=a_1[L(\mu_1)] + \cdots + a_r[L(\mu_r)]$ and $y=b_1[L(\nu_1)] + \cdots + b_s[L(\nu_s)]$ in $\A$; as $\A$ is commutative and $\M$ is totally ordered, one can assume $\mu_1> \cdots > \mu_r$ and $\nu_1> \cdots > \nu_s$.  Then one has
\begin{align*}
\Psi(x+y) &= \Psi(a_1[L(\mu_1)] + \cdots + a_r[L(\mu_r)] + b_1[L(\nu_1)] + \cdots + b_s[L(\nu_s)]) \\
 & \leq  \max \left( \max(  \boldsymbol{\mu}_i , 1 \leq i \leq r) , \max( \boldsymbol{\nu}_j , 1 \leq j \leq s) \right) \\
 &= \max(\Psi(x),\Psi(y)). 
\end{align*}
One also has:
$$ \Psi(xy) = \Psi \left( \sum_{i,j} a_i b_j [L(\mu_i)][L(\nu_j)] \right) . $$
For any $i \geq 2$ (resp. $j \geq 2$), $\mu_i < \mu_1$ (resp. $\nu_j < \nu1$) hence $ \mu_i \odot \nu_j < \mu_1 \odot \nu_1$ if $(i,j) \neq (1,1)$. Hence $\mu_1  \odot  \nu_1$ is the highest word appearing in the decomposition on simples of the above sum. 
Hence 
\begin{align*}
 \Psi(xy) &= \Psi \left( [L(\mu_1)] [L(\nu_1)] \right)) = \varphi (\mu_1 \odot \nu_1)  \\
  & = \varphi(\mu_1) + \varphi(\nu_1) \qquad  \text{by Proposition~\ref{monoid}} \\
  &= \Psi([L(\mu_1)]) + \Psi([L(\nu_1)]) = \Psi(x) + \Psi(y). 
\end{align*}
The remaining axiom in Definition~\ref{defvalu}  is straightforward. 
Finally, recall from Proposition~\ref{monoid} that $\varphi$ is a bijection from $\M$ to $\mathbb{N}^N$. In particular, $\Psi(\A) \supset \varphi(\M) = \mathbb{N}^N$ and thus the value group $\Psi(Frac(\A) \setminus \{ 0 \})$ is indeed the entire group $\zn$. 
\end{proof}

 \begin{lem}
 The valuation $\Psi$ is of maximal rational rank. In particular it is one-dimensional leaves in the sense of \cite{KKh}. 
 \end{lem}
 
 \begin{proof}
As $\A = K_0(\C)$ has a cluster algebra structure, its fraction field is just $\mathbb{Q}(x_1, \ldots , x_N)$ for any cluster $(x_1, \ldots , x_N)$. Thus it is of transcendence degree $N$. By construction, the rational rank of $\Psi$ is also equal to $N$. Thus $\Psi$ is of maximal rational rank. 
 \end{proof}

\begin{rk}
 In fact the valuation $\Psi$ essentially does the same thing as a monomial valuation: up to some automorphism of $\zn$, it can be identified with a valuation coming from the lexicographic order on cluster monomials as in {{\cite[Definition 7.1]{RW}}}. Representation theory provides us with a natural choice of total order on the cluster variables of the initial seed $\mathcal{S}^{\bf w}$: denoting by $\mu_i$ the dominant word such that $x_i = [L(\mu_i)]$, we set
  $$ x_i \lessdot x_j \Leftrightarrow \mu_i < \mu_j . $$
  This induces a total order on Laurent monomials in $x_1, \ldots , x_N$ as in {{\cite[Definition 7.1]{RW}}}. We also denote it $\lessdot$. Then using Corollary~\ref{ordrepara} one can show that for any Laurent monomials ${\bf x}^{\bf a} = \prod_i x_i^{a_i}$ and $ {\bf x}^{\bf b} = \prod_i x_i^{b_i}$ one has
$$ {\bf x}^{\bf b} \lessdot {\bf x}^{\bf b} \Leftrightarrow \Psi({\bf x}^{\bf b}) < \Psi({\bf x}^{\bf b}). $$
\end{rk}

 \bigskip

    \subsection{The grading on $\A$}
     \label{grading}
    
  The grading on $\A$ will essentially be given by the following standard notion of \textit{height} for elements of $Q_{+}$. For any $\beta \in Q_{+}$, write $\beta = \sum_i b_i \alpha_i$. The quantity  
    $$ \hgt(\beta) := \sum_i b_i $$
    is called the \textit{height} of $\beta$. 
   For any $\beta, \gamma \in Q_{+}$, one has $\hgt(\beta + \gamma) = \hgt(\beta) + \hgt(\gamma)$. 
   
\begin{lem} \label{petitlem}
For any word $\nu$ in $\mathcal{M}$, the number of letters of $\nu$  is equal to $\hgt(\wt(\nu))$. 
\end{lem}

Therefore we will denote it using the usual notation for the length of a word namely $|\nu|$. 

\begin{rk}
 In particular, for $\beta \in Q_{+}$, $M$ a simple object in $R(\beta)-gmod$ and $\mu$ the corresponding dominant word, one has $|\mu| = \hgt(\beta)$.
Consider for instance the good Lyndon words ${\bf i}_1 , \ldots , {\bf i}_N$. Then for any $1 \leq k \leq N$, one has 
 $$|{\bf i}_k| = \hgt(\beta_k) . $$
  \end{rk}
  
 Note that  $|\mu \odot \nu| = |\mu| + |\nu|$ for any $\mu, \nu \in \M$. Hence the following definition makes $\A$ into a graded algebra. 
 
     \begin{deftn}[Grading on $\A$] \label{grad}
    To any simple object $M$ in $\C$, we associate the length of the corresponding dominant word, i.e. the integer $|\tilde{\Psi}([M])|$. 
     \end{deftn}
     
   As $\C$ is a monoidal categorification of $\A$, every cluster monomial is a simple object and thus is homogeneous, its degree being the length of the corresponding dominant word. 
   
   \bigskip
    
      \subsection{Newton-Okounkov bodies for $\Cw$}
    \label{NOCw}
    
     We are now ready to construct Newton-Okounkov bodies using the above grading and valuation. For any graded subalgebra $\B$ of $\A = K_0(\C)$, we get a convex compact set $\Delta(\B) \subset \mathbb{R}^N$. Moreover the bodies $\Delta(\B) (\B \subset \A)$ are always contained in $\Delta(\A)$. 
     
    Let us begin with the following statement, that will be useful in the following.
 
 \begin{lem} \label{trivlem}
 Assume $\B$ is a graded finitely generated subalgebra of $\A$. Consider homogeneous generators $b_1, \ldots , b_r$ of $\B$ and set $d_i := \deg b_i$ for every $1 \leq i \leq r$. Assume furthermore that the family $(\Psi(b_1) , \ldots , \Psi(b_r))$ is linearly independent. Then the Newton-Okounkov body $\Delta(\B)$ is the rational  polytope given by 
  $$\Delta(\B) = \text{ConvexHull}(\frac{1}{d_i} \Psi(b_i) , 1 \leq i \leq r) . $$
 \end{lem} 
 
\begin{proof}
Let $f$ be any degree $d$ homogeneous element in $\B$.  We prove that 
$$\frac{1}{d} \Psi(f) \in \text{ConvexHull}(\frac{1}{d_i} \Psi(b_i) , 1 \leq i \leq r) . $$ 
Write $f$ as 
$$ f = \sum_{\substack{(i_1, \ldots , i_r) \\ d_1 i_1 + \cdots + d_r i_r = d}} a_{i_{1}, \ldots, i_{r}} b_1^{i_1} \cdots b_r^{i_r} $$
and decompose each monomial  $b_1^{i_1} \cdots b_r^{i_r} \in \B \subset \A$ on the basis of classes of simple objects in $\C$. As the vectors $\Psi(b_1) , \ldots , \Psi(b_r)$ are linearly independent, one has $\Psi(b_1^{i_1} \cdots b_r^{i_r}) \neq \Psi(b_1^{j_1} \cdots b_r^{j_r})$ if $(i_1, \ldots , i_r) \neq (j_1, \ldots , j_r)$. Thus there is a unique maximal element $\bmu$ among the images $\Psi(b_1^{i_1} \cdots b_r^{i_r})$. This element is then the unique maximal element in the decomposition of $f$ on the basis of classes of simple objects in $\C$. Hence by definition of $\Psi$, one has $\Psi(f) = \bmu =  \Psi(b_1^{i_1} \cdots b_r^{i_r})$ for some  $(i_1, \ldots , i_r) , d_1 i_1 + \cdots + d_r i_r = d$. 
In particular, 
$$ \frac{1}{d} \Psi(f) = \frac{1}{d_1 i_1 + \cdots d_r i_r} (i_1 \Psi(b_1) + \cdots + i_r \Psi(b_r)) \in \text{ConvexHull}(\frac{1}{d_i} \Psi(b_i) , 1 \leq i \leq r) . $$
\end{proof}
  
  Recall that ${\bf e}_k$ stands for the vectors of the standard basis of $\zn$ (see Section~\ref{totord}).
     
      \begin{lem} \label{Deltatot}
      The Newton-Okounkov body $\Delta(\A)$ is given by:
$$ \Delta(\A) = \text{ConvexHull}(\frac{1}{\hgt(\beta_k)} {\bf e}_k , 1 \leq k \leq N). $$
      \end{lem}
      
       \begin{proof}   
       Let $x$ be any cluster variable in $\A$ and $\mu \in \M$ the dominant word such that $x=[L(\mu)]$. We write the canonical factorization of $\mu$ as $\mu = {\bf i}_1^{a_1} \cdots {\bf i}_N^{a_N}$ (see Theorem~\ref{thmKLR} and Remark~\ref{canofacto}). 
       Then by definition $|\mu| = \sum_k a_k \hgt(\beta_k)$. Hence 
       $$ \frac{1}{|\mu|} \Psi(x) = \frac{1}{\sum_k a_k \hgt(\beta_k)} \sum_k a_k {\bf e}_k $$
       which implies
 $$ \frac{1}{|x|} \Psi(x) \in \text{ConvexHull}(\frac{1}{\hgt(\beta_k)} {\bf e}_k , 1 \leq k \leq N). $$ 
    This holds for any cluster variable in $\A$. 
 Now let $\mathfrak{m}=x_1^{a_1} \dots x_r^{a_r}$ be any monomial in the cluster variables (here the $x_i$ are \textit{any} cluster variables, not necessarily of the same cluster) and let $d_i := |x_i|, 1 \leq i \leq r$. One has 
    $$\frac{1}{|\mathfrak{m}|} \Psi(\mathfrak{m}) = \frac{1}{\sum_i a_i d_i} \sum_i a_i \Psi(x_i) = \frac{1}{\sum_i a_i d_i} \sum_i a_i d_i \frac{\Psi(x_i)}{d_i}.$$
  Thus  $\frac{1}{|\mathfrak{m}|} \Psi(\mathfrak{m})$ lies in $\text{ConvexHull}(\frac{1}{\hgt(\beta_k)} {\bf e}_k , 1 \leq k \leq N)$. 
  
 As the valuation of any element $f$ of $\A$ is always equal to the valuation of some monomial as above (see the proof of Lemma~\ref{trivlem}), this proves the desired statement. 
 
       \end{proof}
       
    Note in particular that $\Delta(\A)$ is a simplex of full dimension $N-1$.  
    
    \begin{rk}
    If one chooses a different ordering on $\Phi_{+}^{w}$ (or equivalently a different reduced expression of $w$), then the Newton-Okounkov body $\Delta(\Aqnw)$ will be the same up to some affine isomorphism (whose linear part is given by a permutation of the vectors of the standard basis of $\mathbb{Z}^N$).
     \end{rk}
  
\section{The simplices $\ds$}
 \label{sectionds}

Throughout Sections~\ref{sectionds} and~\ref{hooksection}, we will assume $\mathfrak{g}$ is simply-laced. We will be studying the following Newton-Okounkov bodies: for any seed $\s$ in $\A$ which we write $((x_1, \ldots , x_N),B)$, we consider the graded subalgebra of $\A$ generated  by the cluster variables of $\s$. This is a finitely generated algebra and by Corollary~\ref{invert}, the images of $x_1, \ldots , x_N$ under the valuation $\Psi$ are linearly independent. Hence by Lemma~\ref{trivlem} the corresponding Newton-Okounkov body $\ds$ is the simplex given by 
$$ \ds = \text{ConvexHull}(\frac{1}{|x_i|} \Psi(x_i) , 1 \leq i \leq N) . $$
We begin by outlining the fact that for any seed $\s$, the simplex $\ds$ (as well as $\Delta(\A)$) sits inside an affine hyperplane. This hyperplane is naturally defined from the representation-theoretic data introduced in Section~\ref{background}. Then we use Theorem~\ref{initpara} to prove several properties of the simplices $\ds$. In particular, we exhibit a correspondence between the rational points in $\ds$ and the cluster monomials for the seed $\s$. 

 \subsection{The hyperplane $\hyp$}
 \label{hyperplane}

Recall that  $(e_k , 1 \leq k \leq N)$ stands for the standard basis of $\zn$. We also let $\langle \cdot , \cdot \rangle$ denote the standard Euclidian scalar product on $\mathbb{R}^N$. We let $\boldsymbol{\lambda}$ denote the vector whose $k$th component is the height of the positive root $\beta_k$:
$$ \langle \boldsymbol{\lambda}, {\bf e}_k \rangle := \hgt(\beta_k) $$
for every $k \in \{ 1 , \ldots , N \}$. 
The following Lemma shows that $\boldsymbol{\lambda}$ encodes the grading on $\A$ (see Section~\ref{grading}). 
  
 \begin{lem}
Let $M$ be a simple module in $\C$ and let $\mu \in \M$ the corresponding dominant word. Then one has
 $$ |\mu| =  \langle \boldsymbol{\lambda}, \Psi([M]) \rangle . $$
 \end{lem}
 
 \begin{proof}
  Let us write the canonical factorization of $\mu$ as $\mu = {\bf i}_N^{c_N} \cdots {\bf i}_1^{c_1}$ (see Remark~\ref{canofacto}). Then by definition one has $\Psi(\mu) = ~t(c_1, \ldots , c_N)$. Then using Lemma~\ref{petitlem} we get
$$ |\mu| =  \hgt(\wt(\mu)) =  \hgt(\sum_k c_k \beta_k) = \sum_k c_k \hgt(\beta_k) = \sum_k c_k \langle \boldsymbol{\lambda}, {\bf e}_k \rangle 
 = \langle \boldsymbol{\lambda}, \Psi([L(\mu)]) \rangle . $$
   \end{proof}
   
 Let $\hyp$ denote the affine hyperplane $\{ \langle \boldsymbol{\lambda}, \cdot \rangle = 1 \} \subset \mathbb{R}^N$.
 The following observation is a straightforward consequence of Definition~\ref{grad} and Lemma~\ref{Deltatot}.
 
 \begin{lem}
The simplex $\Delta(\A)$ is contained in $\hyp$. 
\end{lem}
 
  As a consequence, any Newton-Okounkov body associated to a graded subalgebra of $\A$ will also lie in $\hyp$. 
  
  \bigskip

\subsection{First properties of $\ds$}
\label{simplices}

Now we state a couple of algebraic and geometric properties of the simplices $\ds$. We use Theorem~\ref{initpara} as well as a result of Geiss-Leclerc-Schr\"oer ({{\cite[Theorem 8.3]{GLSfactorial}}}). 

 First we exhibit a correspondence between the rational points inside $\ds$ and the monoidal cluster monomials for this seed. By monoidal cluster monomial, we mean an object in $\C$ isomorphic to $ \bigodot_i M_i^{d_i}$ for some nonnegative integers $d_i$, where $M_i$ are the simple modules corresponding to the cluster variables of $\s$. The monoidal categorification statements of \cite{KKKO} imply that monoidal cluster monomials are always real simple objects. 
 
  \begin{prop} \label{propds}
   Let $\s$ be a seed in $\A$ with cluster variables $x_1, \ldots , x_N$. Then for any simple object $M$ in $\C$, one has 
   $$ \text{$M$ is a monoidal cluster monomial for $\s$} \enspace  \Leftrightarrow  \enspace \frac{1}{|M|} \Psi([M]) \in \ds . $$
   Moreover, any rational point in $\ds$ is of the form $ \frac{1}{|M|} \Psi([M])$ for some monoidal cluster monomial $M$ in $\C$. 
   \end{prop}
   
   \begin{proof}
    Let $\mu \in \M$ such that $M \simeq L(\mu)$. We set $d:= |\mu|$ and  $\bmu := \Psi([M]) \in \zn$. Fix a seed $\s := ((x_1 , \ldots , x_N),B)$, and let $\mu_i$ denote  the dominant word such that $x_i = [L(\mu_i)]$. We also set $ d_i := |\mu_i| , \bmu_i := \Psi(x_i) \in \zn$. 
    
   The if part is obvious: if $[L(\mu)]=x_1^{a_1} \cdots x_N^{a_N}$, then in $\M$ one has $\mu = \bigodot_i {\mu_i}^{\odot a_i}$. Hence $ \bmu = \sum_i a_i \bmu_i$ and $ d= \sum_i a_i d_i$. This implies:
    $$\frac{1}{d} \bmu =  \frac{1}{ \sum_i a_i d_i} \sum_i a_i d_i \frac{1}{d_i} \bmu_i \in \ds.$$
    For the only if part, let us write 
    $ \frac{1}{d} \bmu = \sum_i t_i \frac{1}{d_i} \bmu_i $
    with $t_i \geq 0$ for every $i$ and $\sum_i t_i = 1$. Setting $a_i := d t_i / d_i$ for every $1 \leq i \leq N$, this can be rewritten as $ \bmu = \sum_i a_i \bmu_i$ or equivalently  $\bmu = \mathcal{M}_{\s}  ~^t (a_1, \ldots, a_N)$. Now $\bmu \in \zn$ and  $\mathcal{M}_{\s} \in GL_N(\mathbb{Z})$ by Corollary~\ref{invert}. Hence $~^t (a_1, \ldots, a_N) \in \zn$ i.e. $\forall i, a_i \in \mathbb{Z}$ (and hence $a_i \in \mathbb{N}$ as the $a_i$ are positive). This implies 
    $$ \mu = \bigodot_{i} \mu_i^{\odot a_i} $$
    in $\M$ and hence $[M] = \prod_i [L(\mu_i)]^{a_i} = \prod_i x_i^{a_i}$ i.e. $[M]$ is a cluster monomial for the seed $\s$. 
  
   Now let $\boldsymbol{\nu} \in \mathbb{Q}^N \cap \ds$ and write as before $\boldsymbol{\nu} = \sum_i t_i \frac{1}{d_i} \bmu_i $ with $t_i \geq 0$ and $\sum_i t_i = 1$. Then using similar arguments, one shows that there exists a non-negative integer $l$ such that $  \boldsymbol{\mu} := l \boldsymbol{\nu} \in \mathbb{N}^N$ and $ M := L(\mu)$ is a monoidal cluster monomial for the seed $\s$. Now $|M| = l \times \sum_i t_i = l $ and thus one has $ \nu = \frac{1}{l} \boldsymbol{\nu} = \frac{1}{d} \bmu. $
    
   \end{proof}
   
 \begin{cor} \label{interior}
 Let $\s$ and ${\s}'$ be two seeds having different sets of cluster variables. Then the simplices $\ds$ and $\Delta_{\mathcal{S}'}$ have disjoint interiors. 
 \end{cor}
 
  \begin{proof}
  Let $\boldsymbol{\nu} \in \mathbb{Q}^N \cap \ds \cap \Delta_{\mathcal{S}'}$; the previous Proposition implies the existence of a monoidal cluster monomial $M$ (resp.$M'$)  for the seed $\s$ (resp. $\mathcal{S}'$) such that $\frac{1}{d} \bmu = \frac{1}{d'} {\bmu}' = \boldsymbol{\nu}$ (with the same notations as in the previous proof). In particular one has $[M]^{d'} = [M']^{d}$. By  {{\cite[Theorem 8.3]{GLSfactorial}}}, this implies that any cluster variable involved in the monomial $[M]$ has to appear in $[M']$ and vice versa. As by hypothesis $\s$ and $\mathcal{S}'$ have different cluster variables, we conclude that at least one cluster variable of $\s$ does not occur in $[M]$. This obviously implies that $\frac{1}{d} \bmu$ belongs to a face of the simplex $\ds$. In other words $\boldsymbol{\nu} \notin \overset{\circ}{\ds}$.  Hence $\overset{\circ}{\ds} \cap \overset{\circ}{\Delta_{\mathcal{S}'}} = \emptyset$.
  \end{proof}

  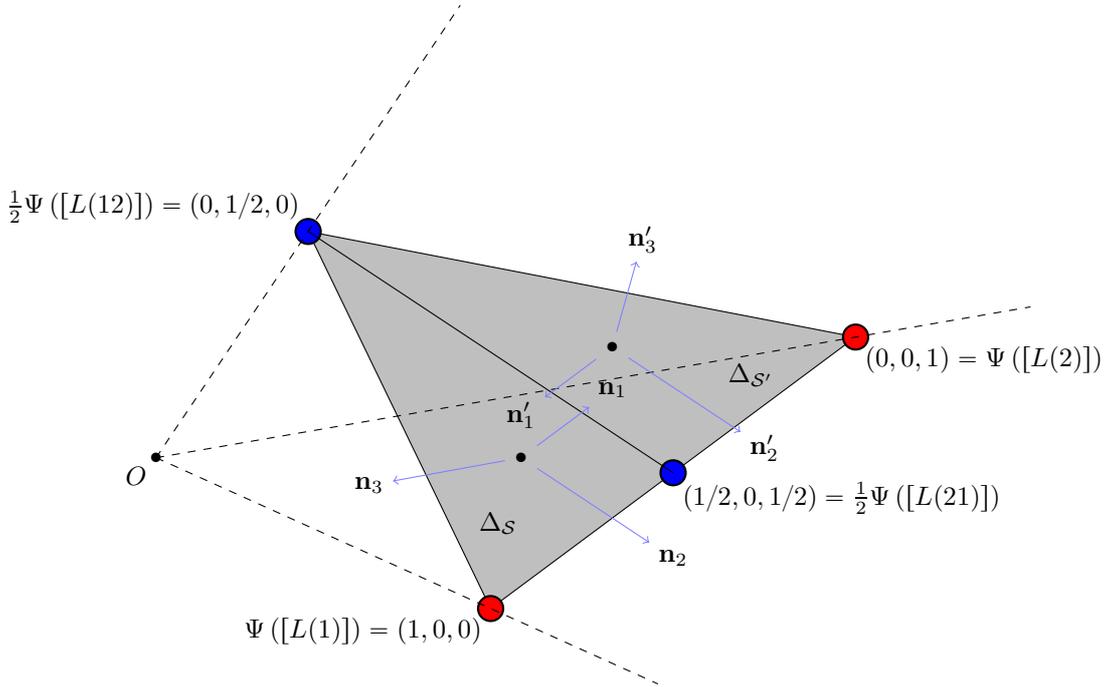
\begin{figure} 
  
   \center
 
  \begin{tikzpicture}[math2d]
  
  \tikzstyle{unfrozen}=[circle,draw,thick,fill=red];
  \tikzstyle{frozen}=[circle,draw,thick,fill=blue];
  \tikzstyle{normal}=[->,blue!50];

 \draw (0,0,0) node[below left] {$O$} node{$\bullet$};
 \draw[fill=gray!50] (0,0,4)--(4,0,0)--(0,2,0)--cycle;
 \draw (4,0,0) node[unfrozen] {} node[below left] {$\Psi \left( [L(1)] \right) =(1,0,0)$}; 
 \draw (0,0,4) node[unfrozen] {} node[below right] {$(0,0,1)= \Psi \left( [L(2)] \right)$}; 
 \draw (2,0,2) node[frozen]{} node[below right] {$(1/2,0,1/2)= \frac{1}{2} \Psi \left( [L(21)] \right)$}; 
 \draw (0,2,0) node[frozen]{} node[above left] {$\frac{1}{2} \Psi \left( [L(12)] \right) = (0,1/2,0)$}; 
 \node (gs) at (13/6,1/2,5/6) {$\bullet$}; 
 \node (gsprime) at (2/3,2/3,2) {$\bullet$}; 
 \node (n1) at (7/6,1/2,11/6) {$\mathbf{n}_1$};
 \node (n2) at (3,-1/3,5/3) {$\mathbf{n}_2$};
 \node (n3) at (17/6,5/6,-1/2) {$\mathbf{n}_3$};
 \node (n1prime) at (5/3,2/3,1) {$\mathbf{n}'_1$};
 \node (n2prime) at (3/2,-1/6,17/6) {$\mathbf{n}'_2$};
 \node (n3prime) at (-2/3,1,8/3) {$\mathbf{n}'_3$};
 
 \draw (2,0,2)--(0,2,0);
 \draw[dashed] (0,0,0)--(0,0,5);
 \draw[dashed] (0,0,0)--(6,0,0);
 \draw[dashed] (0,0,0)--(0,4,0);
 \draw (2,0,1) node[below] {$\Delta_{\mathcal{S}}$};
 \draw (0,0,3.4) node[below] {$\Delta_{\mathcal{S}'}$};
 
 \draw[normal] (gs)--(n1); 
 \draw[normal] (gs)--(n2); 
 \draw[normal] (gs)--(n3); 
  \draw[normal] (gsprime)--(n1prime); 
 \draw[normal] (gsprime)--(n2prime); 
 \draw[normal] (gsprime)--(n3prime); 
 
  \end{tikzpicture}
  
   \caption{In type $A_2$ the simplex $\Delta(\Aqn)$ is covered by the two simplices  $\Delta_{\mathcal{S}^{(1,2,1)}}$ and $\Delta_{\mathcal{S}^{(2,1,2)}}$.}
   
   \label{babyex}
   
    \end{figure}

 \begin{ex} 
 Consider the case where $\mathfrak{g}$ is of type $A_2$ as in Example~\ref{example_sl3}. We choose the natural order $1<2$ which yields a convex order on the set of positive roots given by $\alpha_1 < \alpha_1 + \alpha_2 < \alpha_2$. The corresponding good Lyndon words are respectively given by $1,12$ and $2$. 
   The cluster algebra $\Aqn$ has exactly two seeds, namely $\mathcal{S} = \mathcal{S}^{(1,2,1)}$ and $\mathcal{S}' = \mathcal{S}^{(2,1,2)}$. Each of them contains one unfrozen and two frozen variables. The cluster variables of $\s$ are given by $[L(1)], [L(12)], [L(21)]$ and their respective images under the valuation $\Psi$ are $(1,0,0), (0,1,0), (1,0,1)$ (recall that the last two are the frozen variables). Thus the vertices of $\ds$ have respective coordinates $(1,0,0), (0,1/2,0), (1/2,0,1/2)$ in the $3$-dimensional euclidian space. The unfrozen cluster variable of $\s'$ is given by $[L(2)]$ and $\Psi([L(2)]) = (0,0,1)$. The vertices of $\Delta_{\s'}$ are given by $(0,0,1), (0,1/2,0), (1/2,0,1/2)$. All these points belong to the affine plane $\hyp$ of equation $x_1 + 2x_2 + x_3 = 1$ in $\mathbb{R}^{3}$. Figure~\ref{babyex} shows the simplex $\Delta(\Aqn)$ covered by the two simplices corresponding to the seeds $\mathcal{S}$ and $\mathcal{S}'$. The blue dots correspond to the frozen variables and the red dots to the unfrozen variables. 
 \end{ex}
 
  \begin{ex}
  Consider an underlying Lie algebra $\mathfrak{g}$ of type $A_3$ and let $w := s_1s_2s_3s_1s_2$. Then $\Phi_{+}^{w} = \{ \alpha_1, \alpha_1 + \alpha_2 , \alpha_1 + \alpha_2  + \alpha_3 , \alpha_2 , \alpha_2 + \alpha_3 \}$. The reduced expression of $w$ corresponding to the restriction of the natural ordering $1<2<3$ on $\Phi_{+}^{w}$ is $(1,2,3,1,2)$. By Theorem~\ref{initpara}, the cluster variables of the seed $\mathcal{S}^{(1,2,3,1,2)}$  are given by the following dominant words:
  $$ (1) \qquad (12) \qquad (123) \qquad  (21) \qquad (2312)  $$
  where the first two are unfrozen and the last three are frozen. Their respective images under $\Psi$ are $~^t(1,0,0,0,0)  ~^t(0,1,0,0,0), ~^t(0,0,1,0,0), ~^t(1,0,0,1,0)$ and $ ~^t(0,1,0,0,1)$. The cluster algebra $\Aqnw$ has five seeds. A straightforward computation shows that the other cluster variables of $\Aqnw$ correspond to the dominant words $(2), (23)$ and $(231)$ of respective images $~^t(0,0,0,1,0), ~^t(0,0,0,0,1)$ and $~^t(1,0,0,0,1)$ under $\Psi$. 
  The simplex $\Delta(\Aqnw)$ is of full dimension $4$ inside the affine hyperplane $\hyp \subset \mathbb{R}^{5}$ given by the equation $x_1 + 2x_2 + 3x_3 + x_4 + 2x_5 = 1$. Note that the frozen variable $[L(123)]$ appears in every seed and its image under $\Psi$ is the third vector of the standard basis of $\mathbb{Z}^{5}$ (see also Equation~\eqref{PsiLyndon}). The images under $\Psi$ of the other cluster variables have zero entry along this direction. Hence one can get a three-dimensional picture by projecting on $x_3 = 0$ as shown in Figure~\ref{ex5seeds}. As in the previous example, the blue dots correspond to the frozen variables and the red dots to the unfrozen variables.

    \begin{figure}  
       \begin{tikzpicture}[math3d]
  
  \tikzstyle{unfrozen}=[circle,draw,thick,fill=red];
  \tikzstyle{frozen}=[circle,draw,thick,fill=blue];
  
   \draw[dashed] (0,0,4)--(0,2,0);
  \draw[dashed] (0,0,0)--(2,0,2);
  \draw[dashed] (4/3,0,0)--(2,0,2);
  \draw (0,0,4)--(0,4,0);
  \draw (4,0,0)--(4/3,0,0);
  \draw[dashed] (0,0,4)--(0,0,0);
  \draw (0,0,4)--(2,0,2);
  \draw (4,0,0)--(2,0,2);
  \draw (0,4,0)--(0,2,0);
  \draw (2,0,2)--(0,4,0);
 
  \draw[fill=blue!30,opacity=0.8] (0,0,0)--(4/3,0,0)--(0,2,0)--cycle;
  \draw[dashed,fill=blue!30,opacity=0.9] (2,0,2)--(4/3,0,0)--(0,2,0)--cycle;
   \draw[fill=red!80,opacity=0.5] (4,0,0)--(0,4,0)--(2,0,2)--cycle;
   \draw[fill=red!80,opacity=0.5] (4,0,0)--(0,4,0)--(0,2,0)--cycle;
    \draw[dashed,fill=red!80,opacity=0.1] (2,0,2)--(0,4,0)--(0,2,0)--cycle;
  \draw[dashed,fill=yellow!30,opacity=0.5] (4,0,0)--(2,0,2)--(4/3,0,0)--cycle;
  \draw[fill=yellow!30,opacity=0.5] (4/3,0,0)--(0,2,0)--(4,0,0)--cycle;
   \draw[dashed,fill=yellow!30,opacity=0.7] (4,0,0)--(0,2,0)--(2,0,2)--cycle;
   \draw[dashed,fill=green!60,opacity=0.2] (0,0,0)--(0,0,4)--(2,0,2)--cycle;
  \draw[dashed,fill=green!60,opacity=0.2] (0,2,0)--(0,0,4)--(2,0,2)--cycle;
   \draw[dashed,fill=green!60,opacity=0.2] (0,2,0)--(0,0,0)--(2,0,2)--cycle;
  
  \draw[thick] (4,0,0)--(0,4,0);
   \draw[thick] (4,0,0)--(0,4,0);
  \draw[thick] (4,0,0)--(0,0,0);
  \draw[dashed,thick] (0,0,0)--(0,0,4);
   \draw[thick] (0,4,0)--(0,0,0);
   \draw[thick] (4,0,0)--(0,0,4);
    \draw[thick] (0,0,4)--(0,4,0);
    \draw (4,0,0)--(0,2,0);
  
   \draw (4,0,0) node[unfrozen]{} node[below left,text width=2cm]{ $\Psi \left( [L(1)] \right) = (1,0,0,0,0) $ };
   \draw (0,4,0) node[unfrozen]{} node[above left,text width=2cm]{$ \frac{1}{2} \Psi \left( [L(12)] \right) = (0,1/2,0,0,0)$};
   \draw (0,0,4) node[unfrozen]{} node[below right,text width=2cm]{$ \Psi \left( [L(2)] \right) = (0,0,0,1,0)$};
   \draw (0,0,0) node[unfrozen]{} node[above left,text width=2cm]{$ \frac{1}{2} \Psi \left( [L(23)] \right) = (0,0,0,0,1/2)$};
   \draw (4/3,0,0) node[unfrozen]{} node[below left,text width=2cm]{$ \frac{1}{3} \Psi \left( [L(231)] \right) = (1/3,0,0,0,1/3)$};
   \draw (2,0,2) node[frozen]{} node[below right,text width=2cm]{$ \frac{1}{2} \Psi \left( [L(21)] \right) = (1/2,0,0,1/2,0)$};
   \draw (0,2,0) node[frozen]{} node[above left,text width=2cm]{$ \frac{1}{4} \Psi \left( [L(2312)] \right) = (0,1/2,0,0,1/4) \enspace $} ;
   
    \end{tikzpicture}
    
     \bigskip
     \bigskip
     
      \caption{In type $A_3$ with $w := s_1s_2s_3s_1s_2$, the simplex $\Delta(\Aqnw)$ is covered by five smaller simplices, colored in white, red, yellow, blue, and green.}
   
    \label{ex5seeds}
   
    \end{figure}
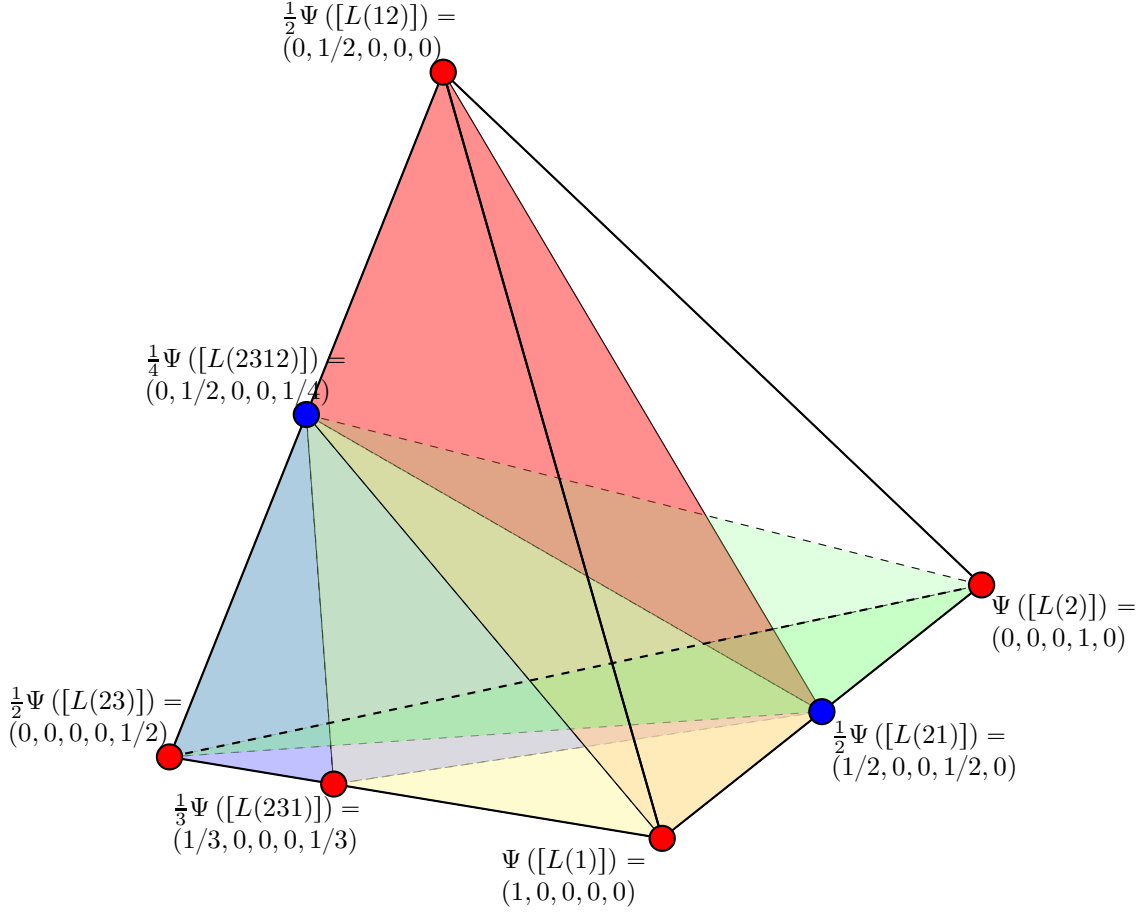
    
     \end{ex}
    
    \bigskip
    
\subsection{Tropical $\epsilon$-mutation} 
 \label{epsilonmut}
 
 Tropical \textit{epsilon}-mutations were defined in \cite{Nakanishi} as tropical exchange relations similar to the mutation rules of $g$-vectors or $c$-vectors. They involve a sign $\epsilon$ which can be chosen arbitrarily at each mutation. The usual tropical exchange relations of $g$-vectors and $c$-vectors are obtained by choosing $\epsilon$ to be the coherency sign of $c$-vectors. In this section we show how monoidal categorifications can provide new examples of interesting tropical $\epsilon$-mutations. In the case of the categorifications of coordinate rings via finite type quiver Hecke algebras, the sign $\epsilon$ comes from the natural ordering on parameters of simple objects in $R-gmod$.

  We fix a seed $\s = ((x_1, \ldots , x_N), B)$ in $\A$ as well as a mutation direction $k$. Let $ \mathcal{S}'$ denote the seed obtained from $\s$ after mutation in the direction $k$. The cluster variable $x_k$ is replaced by $x'_k$. We let $M_i$ (resp. $M'_k$) denote the simple module whose class is $x_i$  for every $1 \leq i \leq N$ (resp. $x'_k$). It follows from the constructions of \cite{KKKO} that this cluster mutation at the level of $\A$ comes from a short exact sequence 
 \begin{equation} \label{ses}
     0 \rightarrow  \bigodot_{b_{ik}>0} M_i^{\circ b_{ik}} \rightarrow  M_k \circ M'_k \rightarrow \bigodot_{b_{ik}<0} M_i^{\circ (-b_{ik})} \rightarrow 0
   \end{equation}
  in $\C$.
 We set 
  $$ \bmu_i := \Psi(x_i), \quad i \in J  \qquad \text{and} \qquad \bmjh := \Psi(\yjh) , \quad j \in J_{ex} $$
 (see Section~\ref{remindCasbi} for the definition of the $\yjh$). 
 As the exchange matrix $B$ has full rank $\sharp J_{ex}$, one can extend it into a $N \times N$ invertible matrix $\tilde{B}$. We then set as before
  $$  \yjh := \prod_{1 \leq i \leq N} x_i^{b_{ij}} \quad \text{and} \quad \bmjh := \Psi(\yjh) \quad \text{for any $1 \leq j \leq N$} . $$
  We show that the vectors $\bmjh$ satisfy a tropical $\epsilon$-mutation in the sense of Nakanishi \cite{Nakanishi}.
 
   As $\zn$ is equipped with a total lexicographic order, the following definition makes sense: for any $k \in J_{ex}$, we set 
    $$ \eta_k := \begin{cases} 
              +1 & \text{ if $\bmkh > \overrightarrow{0} $ } \\
              -1 & \text{ otherwise} 
                \end{cases} . $$
 
   \begin{lem} \label{epsilonmutation}
 The following equality holds in $\zn$:
  $$ {\bmu}'_k =  - \bmu_k +  \sum_{i} [\eta_k b_{ik}]_{+} \bmu_i . $$
  \end{lem} 

 \begin{proof}
Consider the exchange relation 
$$ x'_k x_k  =  \prod_{i, b_{ik}>0} x_i^{b_{ik}} + \prod_{i, b_{ik}<0} x_i^{-b_{ik}}. $$
By Lemma~\ref{valu}, $\Psi$ is a valuation hence we get
 $$ \Psi(x'_k) = - \Psi(x_k) + \max \left( \sum_{b_{ik} > 0} b_{ik} \Psi(x_i) , \sum_{b_{ik} < 0} (- b_{ik}) \Psi(x_i) \right).$$
   Moreover, 
$$  \sum_{b_{ik} > 0} b_{ik} \Psi(x_i) > \sum_{b_{ik} < 0} (- b_{ik}) \Psi(x_i)
   \Leftrightarrow \sum_{i} b_{ik} \Psi(x_i) > \overrightarrow{0} 
    \Leftrightarrow \Psi(\yjh) > \overrightarrow{0} 
    \Leftrightarrow \eta_k = +1 $$
     Hence if $\eta_k = 1$, then $ {\bmu}'_k =  \Psi(x'_k) = - \Psi(x_k) + \sum_{b_{ik} > 0} b_{ik} \Psi(x_i) = - \bmu_k + \sum_{i} [b_{ik}]_{+} \bmu_i $. 
   Similarly one can show that if $\eta_k = -1$ then ${\bmu}'_k = \Psi(x'_k) = - \Psi(x_k) + \sum_{i} [-b_{ik}]_{+} \Psi(x_i) = - \bmu_k + \sum_{i} [-b_{ik}]_{+} \bmu_i$.
\end{proof}   
  
   Now we describe the mutation of the vectors $\bmjh$. We let ${\bmjh}'$ denote the analogs of $\bmjh$  for the seed $\mathcal{S}'$. 
   
   \begin{lem} \label{tropical}
    The vectors ${\bmjh}', 1 \leq j \leq N$ are given by:
$${\bmjh}' = 
\begin{cases}
  - \bmkh & \text{if $j=k$,} \\
 \bmjh + [\eta_k b_{jk}]_{+} \bmkh &\text{otherwise.}
  \end{cases} $$
  \end{lem}
   
    \begin{proof}
    By {{\cite[Proposition 3.9]{FZ4}}}, one has 
     $$   {\yjh}' = \begin{cases}
           {\ykh}^{-1} & \text{ if $j=k$, } \\
           \yjh  {\ykh}^{[b_{kj}]_{+}}  (\ykh + 1)^{-b_{kj}} & \text{otherwise.}
            \end{cases}$$
  Hence applying the valuation $\Psi$ yields
 $$   {\bmjh}' = \begin{cases}
         - {\bmkh}  & \text{ if $j=k$, } \\
           \bmjh +  [b_{kj}]_{+} \bmkh - b_{kj} \left( \max (\bmkh , \overrightarrow{0}) \right) & \text{otherwise.}
            \end{cases}$$
    Consider $j \neq k$. If $\eta_k = 1$ then $\bmkh > \overrightarrow{0}$ and thus
    $${\bmjh}' = \bmjh +  [b_{kj}]_{+} \bmkh - b_{kj} {\bmkh} = \bmjh + [-b_{kj}]_{+} \bmkh = \bmjh + [b_{jk}]_{+} \bmkh . 
    $$ 
    Similarly if $\eta_k = -1$ then 
 $$ {\bmjh}' = \bmjh +  [b_{kj}]_{+} \bmkh  = \bmjh +  [-b_{jk}]_{+} \bmkh . $$
    \end{proof}

 \begin{rk}
   As explained in \cite{Nakanishi}, tropical $\epsilon$-mutations are of particular interest in cluster theory for describing mutation rules of $g$-vectors and $c$-vectors. This follows from a particular choice of tropical sign, namely the sign of $c$-vectors. This sign-coherency property has been proved by Derksen-Weyman-Zelevinsky \cite{DWZ} in the skew-symmetric case and by Gross-Hacking-Keel-Kontsevich \cite{GHKK} in the general case. In our setting, the sign in the mutation rule does not come from the sign-coherence of $c$-vectors. Thus the mutation rule given by Lemma~\ref{tropical} gives a new example of $\epsilon$-tropical mutation where the tropical sign $\eta_k$ encodes the natural ordering on dominant words.
 \end{rk} 
   
   We end this section with a couple of elementary remarks that will be useful in Section~\ref{normalfan}.
   
    \begin{lem} \label{free}
 The vectors $\bmjh , j \in J$ form a basis of $\mathbb{R}^N$. 
 \end{lem}
 
 \begin{proof}
 Recall the matrix $M_{\s}$ introduced in Section~\ref{seedsection}. By definition, one has 
 $$ \forall j \in J, \bmjh = \Psi(\yjh) = \sum_{1 \leq i \leq N} \tilde{b}_{ij} \Psi(x_i) = M_{\s}\tilde{B} {\bf e}_j  $$
 where $\tilde{B}$ is the extended exchange matrix of the seed $\s$. By Corollary~\ref{invert} the matrix $M_{\s}$ is invertible. As $\tilde{B}$ is invertible, the composition $M_{\s}B$ is invertible as well and thus the family $(\bmjh)_{j \in J}$ is a basis.  
 \end{proof}
 
 \begin{rk} \label{bmjhvect}
  By construction the three non trivial terms in the short exact sequence~\eqref{ses} have the same weight. This implies that for any $j \in J_{ex}$, the vector $\bmjh$ is of weight zero. By this we mean that for $j \in J_{ex}$, $\bmjh$ belongs to the kernel of the following linear map:
  $$ \begin{array}{cccc}
   {\bf \wt} : &\mathbb{R}^N & \longrightarrow & \mathbb{R}^{\sharp I} \\
    {}    &     {\bf c} = ~^t(c_1, \ldots , c_N) & \longmapsto & \sum_k c_k \beta_k
  \end{array} $$ 
 where elements of $Q_{+}$ are identified with vectors in $\mathbb{R}^{\sharp I}$ in an obvious way. Moreover $\sharp I = \sharp J_{fr} = N - \sharp J_{ex}$ (see {{\cite[Section 2.3]{KK}}}). Hence the $\bmjh , j \in J_{ex}$ form a basis of $\ker {\bf \wt}$. This holds for any seed. 
 \end{rk}
 
  \bigskip

 \subsection{The normal fan to $\ds$}
  \label{normalfan}
  
  In this section, we provide a geometric interpretation of certain tools used in our previous work \cite{Casbi} as analogues in the group $\G_w$ of Fomin-Zelevinsky's variables $\yjh$ with respect to the cluster structure of $\Aqnw$. More precisely, we show that, up to some universal linear transformation, the images under the valuation $\Psi$ of the $\yjh$ for each seed $\s$ of $\Aqnw$ form a face of the normal fan of the simplex $\ds$. 
  
   Recall that for a polytope $\Delta$ in a (finite-dimensional) Euclidian space and for each face $F$ of $\Delta$, the normal cone of $F$ is defined as the cone whose generating rays are the outer normal vectors to the facets (i.e. the faces of codimension $1$) containing $F$. The collection of all these cones is the normal fan of $\Delta$. The rays (i.e. the faces of dimension $1$) of the normal fan of $\Delta$ are precisely the normal vectors to the facets of $\Delta$. 
  
  Recall that we fix a total order $<$ on $I$, an element $w \in W$ and the unique reduced expression ${\bf w}$ of $w$ corresponding to the restriction of  $<$ to $\Phi_{+}^{w}$. 
Let us introduce a family of vectors ${\bf n}_j^{\s}$ for every seed $\s$ in $\A$. They are defined inductively as follows:

 \begin{deftn} \label{defnvector}
 Consider the initial seed $\mathcal{S}^{\bf w}$. We define 
 $$ {\bf n}_j^{0} := 
 \begin{cases} {\bf e}_j - {\bf e}_{j_{+}} -  (\hgt(\beta_j) - \hgt(\beta_{j_{+}})) \frac{\boldsymbol{\lambda}}{||\boldsymbol{\lambda}||^{2}} & \text{if $j \in J_{ex}$,} \\
                          {\bf e}_j - \hgt(\beta_j) \frac{\boldsymbol{\lambda}}{||\boldsymbol{\lambda}||^{2}} & \text{if $j \in J_{fr}$.}
 \end{cases}
 $$
 Then given two seeds $\s,{\s}'$ related to each other by a mutation in the direction $k \in J_{ex}$, the vectors ${\bf n}_j^{{\s}'}$ are related to the ${\bf n}_j^{\s}$ by the following tropical $\epsilon$-mutation:
  $$ {\bf n}_j^{{\s}'}  = 
\begin{cases}
  - {\bf n}_k^{\s} & \text{if $j=k$,} \\
{\bf n}_j^{\s} + [\eta_k b_{jk}]_{+} {\bf n}_k^{\s} &\text{otherwise.}
  \end{cases} $$
 \end{deftn}
 
 As any seed can be reached by a finite sequence of mutations from the initial seed $\mathcal{S}^{\bf w}$, this defines in a unique way the vectors ${\bf n}_j^{\s}$ for every seed $\s$. 
 The main result of this section can now be stated as follows:
 
 \begin{thm} \label{mainthm}
 For any seed $\s$, the vectors ${\bf n}_j^{\s}$ are the rays of the normal fan of $\ds$. 
 \end{thm}

Let $\s$ be an arbitrary seed in $\A$. The simplex $\ds$ is of full dimension $N-1$ inside $\hyp$. Hence for every vertex $P_j$ of $\ds$, one can consider the facet $F_j$ of $\ds$ which does not contain the vertex $P_j$. For every $1 \leq j \leq N$, consider the linear hyperplane of $\mathbb{R}^N$ containing the points $P_i , i \neq j$. One considers the unique vector ${\bf N}_j^{\s}$ normal to this hyperplane such that  $\langle \bmu_j , {\bf N}_j^{\s}  \rangle = 1$ (recall that the $\bmu_j^{\s}$ are linearly independent by Lemma~\ref{free}).

 First we show that the vectors ${\bf N}_j^{\s}$ satisfy the suitable tropical mutation rule:
 
  \begin{lem} \label{mutNvector}
The vectors ${\bf N}'_j  , 1 \leq j \leq N$ are given by:
$$ {\bf N}'_j  = 
\begin{cases}
  - {\bf N}_j & \text{if $j=k$,} \\
 {\bf N}_j  + [\eta_k b_{jk}]_{+} {\bf N}_k &\text{otherwise.}
  \end{cases} $$
  \end{lem}

 \begin{proof}
  The simplices $\ds$ and $\Delta_{\mathcal{S}'}$ share the facet $F_k$ consisting of the points $P_i, i \neq k$.  By definition, both ${\bf N}_k$ and ${\bf N}'_k$ are orthogonal to the linear hyperplane containing the points $P_i, i \neq k$. Hence ${\bf N}_k = c_k {\bf N}'_k$ for some (nonzero) real scalar $c_k$. Now,
 $$ 1 = \langle {\bmu}'_k, {\bf N}'_k \rangle
    = - \langle \bmu_k, {\bf N}'_k \rangle + \sum_{i, sgn(b_{ik})=\eta_k} b_{ik} \langle   \bmu_i, {\bf N}'_k \rangle  
    = - \langle  \bmu_k , {\bf N}'_k \rangle
    = - c_k $$
  which proves the mutation relation for ${\bf N}_k$. 
  
   Now let $j \neq k$. One has,
   \begin{align*}
  \langle  \bmu_k , {\bf N}'_j \rangle
   &= - \langle  {\bmu}'_k , {\bf N}'_j \rangle + \sum_{i} [\eta_k b_{ik}]_{+} \langle \bmu_i , {\bf N}'_j \rangle  \quad \text{by Lemma~\ref{epsilonmutation},} \\
   &= \sum_{i} [\eta_k b_{ik}]_{+} \langle \bmu_i , {\bf N}'_j \rangle  
   = [\eta_k b_{jk}]_{+} .   
   \end{align*}
 Thus the vectors ${\bf N}_j$ and ${\bf N}'_j - [\eta_k b_{jk}]_{+} {\bf N}_k$ are orthogonal to the linear hyperplane $\Vect_{\mathbb{R}}(\bmu_i , i \neq j)$.
    Hence one can write 
   $${\bf N}'_j = c_j {\bf N}_j + [\eta_k b_{jk}]_{+} {\bf N}_k$$ 
   for some (nonzero) real scalar $c_j$. Computing the scalar product of both hand sides with $\bmu_j$ gives $c_j = 1$ which finishes the proof. 
 \end{proof}
 
 Now we relate the ${\bf N}_j^{\s}$ to the ${\bf n}_j^{\s}$, beginning with the initial seed $\mathcal{S}^{\bf w}$. 

\begin{prop} \label{Nisninit}
Consider the seed $\mathcal{S}^{\bf w}$. Then for any $j \in J$ on has
 $$ {\bf n}_j^{0} = {\bf N}_j^{0} - \frac{\langle \boldsymbol{\lambda} , {\bf N}_j^{0}  \rangle}{||\boldsymbol{\lambda}||^{2}} \boldsymbol{\lambda} . $$
 \end{prop}
 
 \begin{proof}
For simplicity we write ${\bf N}_j$ (resp. ${\bf n}_j$) for ${\bf N}_j^{0}$ (resp. ${\bf n}_j^{0}$) throughout this proof. 
Recall from Section~\ref{seedsection} that we set $J_{k} := \{j \leq k | i_j=i_k \} = \{j_0=k > j_1 > \cdots > j_{r_k} \}$. 
 
 \smallskip
 
 {\bf First case: $j \in J_{ex}$.} First consider $k \in J \setminus J_{j}$. By Theorem~\ref{initpara}, the dominant word $\mu_k$ is the concatenation in the decreasing order of the good Lyndon words ${\bf i}_l , l \in J_{k} , l \leq k$. By definition, ${\bf N}_j$ is orthogonal to every $\bmu_i, i \neq j$. 
 In particular, it is orthogonal to $\bmu_{j_{r_k}} , \bmu_{j_{r_k -1}} , \ldots ,  \bmu_k$. Thus one has
  $$ 0 =  \langle \bmu_{j_{r_k}} , {\bf N}_j \rangle =  \langle {\bf e}_{j_{r_k}} , {\bf N}_j \rangle $$
  and 
 $$  0 =  \langle \bmu_{j_{r_k -1}} , {\bf N}_j \rangle =  \langle {\bf e}_{j_{r_k -1}} +   \bmu_{j_{r_k}} , {\bf N}_j \rangle =  \langle {\bf e}_{j_{r_k -1}} , {\bf N}_j \rangle .  $$
 By a straightforward induction, this implies that the $l$th component of ${\bf N}_j$ is zero for every $l \in J_{k}$. This holds for every $k \notin J_{j}$. 
 
 Similar arguments show that the $l$th component of ${\bf N}_j$ is zero for every $l \in J_{j}$ with $l<j$. By definition, one has $\langle \bmu_j , {\bf N}_j \rangle = 1$. Hence 
 $$ 1 =  \langle \bmu_{j} , {\bf N}_j \rangle  = \langle {\bf e}_{j} +   \bmu_{j_{-}} , {\bf N}_j \rangle =   \langle {\bf e}_{j} , {\bf N}_j \rangle $$
 and thus the $j$th component of ${\bf N}_j$ is equal to $1$. Now as $j$ is assumed to lie in $J_{ex}$, one has $j_{+} \leq N$ (see Section~\ref{remindKK}). Hence one can write 
 $$ \langle \bmu_{j_{+}} , {\bf N}_j \rangle = 0 \quad \text{with} \quad \mu_{j_{+}} = {\bf i}_{j_{+}} {\bf i}_{j} \cdots {\bf i}_{r_{j}} $$
 by Theorem~\ref{initpara}. Thus 
 $$ 0 = \langle \bmu_{j_{+}} , {\bf N}_j \rangle =  \langle {\bf e}_{j_{+}} +   \bmu_{j} , {\bf N}_j \rangle =  \langle  {\bf e}_{j_{+}} , {\bf N}_j \rangle + 1 $$
 by definition of ${\bf N}_j$. Hence the $j_{+}$th entry of ${\bf N}_j$ is $-1$. Then a straightforward induction similar to the first case shows that the $l$th component of ${\bf N}_j$ is zero if $l>j_{+}$. 
 
  Thus we have shown that ${\bf N}_j$ has exactly two non zero entries, namely the $j$th equal to $1$ and the $j_{+}$th equal to $-1$. Hence 
  $$ {\bf N}_j - \frac{\langle \boldsymbol{\lambda} , {\bf N}_j  \rangle}{||\boldsymbol{\lambda}||^{2}} \boldsymbol{\lambda} = {\bf e}_j - {\bf e}_{j_{+}} -  (\hgt(\beta_j) - \hgt(\beta_{j_{+}})) \frac{\boldsymbol{\lambda}}{||\boldsymbol{\lambda}||^{2}} . $$
   Comparing with Definition~\ref{defnvector}, we conclude that the desired statement holds.
  
  \smallskip
  
  {\bf Second case: $j \in J_{fr}$.} One shows as before that the $k$th entry of ${\bf N}_j$ is zero for $k \notin J_{j}$. In this case $J_{j}$ is exactly the set of all indices of occurrences of the letter $j$ in the chosen reduced expression of $w$. Writing  $\langle \bmu_{i} , {\bf N}_j \rangle = 0$ for every $i \in J_{j} \setminus \{j \}$ implies that all the entries of ${\bf N}_j$ are zero except the $j$th. This entry is equal to  $\langle \bmu_{j} , {\bf N}_j \rangle$ which is $1$ by definition. Hence ${\bf N}_j = {\bf e}_j$ for every $j \in J_{fr}$. This implies 
  $$ {\bf N}_j - \frac{\langle \boldsymbol{\lambda} , {\bf N}_j  \rangle}{||\boldsymbol{\lambda}||^{2}} \boldsymbol{\lambda}  =  {\bf e}_j - \hgt(\beta_j) \frac{\boldsymbol{\lambda}}{||\boldsymbol{\lambda}||^{2}} = {\bf n}_j $$
   for every $j \in J_{fr}$ which finishes the proof. 

 \end{proof}
 
 Now we show that the statement of Proposition~\ref{Nisninit} holds for every seed: 
 
  \begin{lem} \label{Nisn}
  Let $\s$ be any seed in $\A$, and let $j \in J$. Then one has 
  $$ {\bf n}_j^{\s} = {\bf N}_j^{\s} - \frac{\langle \boldsymbol{\lambda} , {\bf N}_j^{\s}  \rangle}{||\boldsymbol{\lambda}||^{2}} \boldsymbol{\lambda}  . $$
  \end{lem}
  
   \begin{proof}
   For the seed $\mathcal{S}^{\bf w}$, it follows from Proposition~\ref{Nisninit}. As the function $ \hgt(\cdot)$ is linear, Lemma~\ref{mutNvector} shows that the $\tilde{{\bf N}_j}^{\s}$ follows the same tropical mutation rule as the ${\bf n}_j$. Hence by induction the equality holds for every seed. 
   \end{proof}
   
  One can now finish the proof of Theorem~\ref{mainthm}:
 
  \begin{proof}[Proof of Theorem~\ref{mainthm}.]
 For any $j \in J$, the vector ${\bf N}_j^{\s}$ is orthogonal to every $\bmu_i$, $i \neq j$ hence to the vectors  $\frac{\bmu_p}{\hgt(\mu_p)} - \frac{\bmu_q}{\hgt(\mu_q)}$ for any $p,q \neq j$. These generate the underlying linear space of $F_j$ and hence ${\bf N}_j^{\s}$ is orthogonal to $F_j$. Moreover, the facet $F_j$ is contained in $\hyp$ for any $j \in J$. Hence $\boldsymbol{\lambda}$ is orthogonal to $F_j$. 
 The conclusion follows from Lemma~\ref{Nisn}.
 \end{proof}
 
 Let us finish this section by explaining why Theorem~\ref{mainthm} provides an explicit geometric realization of the cluster-theoretic dominance order (see Definition~\ref{defdom}). Fix a seed $\s =((x_1, \ldots , x_N) , B)$ in $\A$. The dominance order was introduced by F.Qin as a partial ordering on Laurent monomials in $x_1, \ldots , x_N$ in the study of common triangular bases for (quantum) cluster algebras. The vectors $\bmjh$ (see Section~\ref{epsilonmut}) were defined in \cite{Casbi} as a natural analog of this order at the level of parameters for simple modules in $\C$. More precisely let  $\overrightarrow{\mathcal{N}^{\s}}$ denote the linear convex cone generated by the $\bmjh , j \in J_{ex}$. For every simple object $M$ in $\C$,  let $\mathcal{N}_{M}^{\s}$ denote the affine cone with origin $\Psi([M])$ and direction $\overrightarrow{\mathcal{N}^{\s}}$; then one can see that the simple objects in $\C$ whose classes are smaller than $[M]$ for $\preccurlyeq_{\s}$ correspond to the integral points of $\mathcal{N}_{M}^{\s}$.  
 
 Theorem~\ref{mainthm} can now be reformulated as follows.
Consider the vector subspace ${\bf V}$ of $\mathbb{R}^N$ generated by the ${\bf n}_j^{0} , j \in J_{ex}$. By Remark~\ref{bmjhvect}, the $\bmjh , j \in J_{ex}$ form a basis of $\ker {\bf \wt}$ hence one can consider the unique linear map $T$ defined as
   $$ \begin{array}{cccc}
   T :  &\ker \left( {\bf \wt} \right) &\longrightarrow & {\bf V} \\
   {} &   \bmjh & \longmapsto &  {\bf n}_j^{0}
\end{array}  $$
for every $j \in J_{ex}$. 
  
  \begin{cor} \label{cordominance}
  The map $T$ is a linear isomorphism and for every seed $\s$, the image under $T$ of the cone $\overrightarrow{\mathcal{N}^{\s}}$ is the face ${\bf n}_j^{\s} , j \in J_{ex}$ of the normal fan of $\ds$.
   \end{cor}
  
\begin{proof}
  The vectors ${\bf n}_j^{0} , j \in J_{ex}$ form a basis of ${\bf V}$ hence $T$ is an isomorphism. 
  By Lemma~\ref{tropical}, the vectors $\bmjh$ follow the same tropical mutation rule as the ${\bf n}_j$. Hence one has 
 $$  \forall j \in J_{ex}, T \bmjh = {\bf n}_j $$
 for every seed $\s$. This is the desired statement by Theorem~\ref{mainthm}. 
  \end{proof}

\begin{rk}
Definition~\ref{defnvector} allows us to get an explicit description  of the subspace ${\bf V}$. This subspace essentially describes the  exchange part of the cluster algebra $\A$. For instance for $w=w_0$ and ${\bf w} = (1,2,1,3,2,1,\ldots , n , \ldots 1)$ in type $A_n$, {{\cite[Theorem 6.1]{Casbi}}} implies that $\hgt(\beta_j) = \hgt(\beta_{j_{+}})$ for every $j \in J_{ex}$. Hence for every $j \in J_{ex}$, the vector ${\bf n}_j^{0} , j \in J_{ex}$ is simply ${\bf e}_j - {\bf e}_{j_{+}}$. Let $M_1, \ldots , M_n$ denote the simple modules corresponding to the frozen variables in $\A$.  Then using Theorem~\ref{initpara}, one can check that in this case ${\bf V}$ is exactly the orthogonal of the vector subspace generated by $\Psi([M_1]) , \ldots , \Psi([M_n])$. 
\end{rk} 
  
   \begin{ex}
   Figure~\ref{babyex} shows the normal vectors to each facet of the simplices $\ds$ and $\Delta_{\s'}$. In this case there is only one direction of mutation, and $\mathbf{n}_1$ (resp. $\mathbf{n}'_1$) corresponds to $\hat{\boldsymbol{\mu}_1}$ for the seed $\s$ (resp. $\s'$) up to some non zero scalar. 
    \end{ex}

\section{Towards colored hook formulas}
 \label{hooksection}

In this section we focus on the case where $\Aqnw$ is a cluster algebra of finite type, i.e. there is a finite number of seeds. We obtain an equality between rational functions involving the weights of the simple modules corresponding to the cluster variables of $\Aqnw$. 
As a consequence, we get a cluster-theoretic formula for the quantity $\frac{N!}{\prod_{\beta \in \Phi_{+}^{w}} \hgt(\beta)}$. This quantity has a well-known significance in combinatorics and Lie theory: Peterson-Proctor related this quantity to the combinatorics of $d$-complete posets (see \cite{P1,P2}) . Under some technical assumption on $w$ ($w$ is assumed to be \textit{dominant minuscule} in the terminology of \cite{Stem}) they prove that this quantity is exactly the number of reduced expressions of $w$. This Peterson-Proctor hook formula is also related to the dimension of certain remarkable simple representations of quiver Hecke algebras constructed by Kleshchev-Ram, see {{\cite[Theorem 3.10]{KRhom}}}.

 In \cite{Nakada}, Nakada proposed a generalization of the Peterson-Proctor hook formula. Recall from Section~\ref{remindKLR} that $\alpha_1, \ldots , \alpha_n$ stand for the simple roots of $\mathfrak{g}$. One considers the $\alpha_i, 1 \leq i \leq n$ as formal variables and we let $\mathbb{L}$ denote the field $\mathbb{C}(\alpha_1, \ldots, \alpha_n)$. For every $\beta = \sum_i a_i \alpha_i \in Q_{+}$, one associates a formal rational function 
$$ \frac{1}{\beta} := \frac{1}{a_1 \alpha_1 + \cdots + a_n \alpha_n} \in \mathbb{L}. $$  
Specializing the $\alpha_i$ to $1$, the value of this rational function is exactly $\frac{1}{\hgt(\beta)}$. 
Then Nakada proves the following \textit{colored hook formula}:

\begin{thm}[{{\cite[Corollary 7.2]{Nakada}}}] \label{HookNakada}
 Assume $w$ is a dominant minuscule element of $W$ in the terminology of \cite{P1,P2,Stem}. Recall that $N$ denotes the length of $w$. Then the following equality holds in $\mathbb{L}$:
  \begin{equation} \label{eqNakada}
   \prod_{\beta \in \Phi_{+}^{w}} \frac{1}{\beta} 
 =  \sum_{(i_1, \ldots , i_N) \in \text{MPath(w)}} \frac{1}{\alpha_{i_1}} \frac{1}{\alpha_{i_1} + \alpha_{i_2}} \cdots \frac{1}{\alpha_{i_1} + \alpha_{i_2} + \cdots + \alpha_{i_N}}  
   \end{equation} 
 \end{thm}
 
  The set $\text{MPath}(w)$ is a finite set in bijection with the set of all reduced expressions of $w$. We refer to {{\cite[Sections 2,7]{Nakada}}} for more details. Every term of the sum in the right hand side of Equation~\eqref{eqNakada} is equal to $1/N!$ when specializing the $\alpha_i$ to $1$. Hence as an immediate consequence of this result, one gets that the cardinal of $\text{MPath}(w)$ coincides with the Peterson-Proctor hook formula 
 $$ \sharp \text{MPath}(w) = \frac{N!}{\prod_{\beta \in \Phi_{+}^{w}} \hgt(\beta)} . $$
  
  \begin{rk}
  In fact, the main result of \cite{Nakada} expresses the rational function $\prod_{\beta \in \Phi_{+}^{w}} \left( 1 + \frac{1}{\beta} \right)$ as a sum of rational functions of the form 
  $$\frac{1}{\alpha_{i_1}} \frac{1}{\alpha_{i_1} + \alpha_{i_2}} \cdots \frac{1}{\alpha_{i_1} + \alpha_{i_2} + \cdots + \alpha_{i_l}} $$
   with $l \leq N$, where the tuples $(i_1, \ldots , i_l)$ run over a set $\text{Path}(w)$ strictly containing $\text{MPath}(w)$. The equality given by Theorem~\ref{HookNakada} is obtained by considering the terms of lowest degree. 
  \end{rk}
  
For every seed $\s = ((x_1^{\s}, \ldots , x_N^{\s}),B^{\s})$ in $\Aqnw$ and any $1 \leq j \leq N$, consider the unique dominant word $\mu_j^{\s}$ such that $x_j^{\s}= [L(\mu_j^{\s})]$. We write the weight of $\mu_j^{\s}$ as $\wt(\mu_j^{\s}) =  \sum_i a_{i,j}^{\s} \alpha_i$ (see Section~\ref{remindKLR}). Then mimicking \cite{Nakada}, one considers the rational function 
  $$ \frac{1}{\wt(\mu_j^{\s})} := \frac{1}{a_{1,j}^{\s} \alpha_1 + \cdots + a_{n,j}^{\s} \alpha_n} \in \mathbb{L}. $$
We can now state the main result of this section. 

\begin{thm} \label{prophook}
Assume $w \in W$ is such that the cluster algebra $\Aqnw$ is of finite type. Then the following equality holds in $\mathbb{L}$: 
\begin{equation} \label{eqprophook}
 \prod_{\beta \in \Phi_{+}^{w}} \frac{1}{\beta} = \sum_{\s} \prod_{1 \leq j \leq N} \frac{1}{\wt(\mu_j^{\s})} . 
 \end{equation}
\end{thm}

We fix $w \in W$ and we write as in Section~\ref{remindKK} $\Phi_{+}^{w} = \{ \beta_1 < \cdots < \beta_N \}$. We identify positive roots in $\Phi_{+}^{w}$ with elements of $\mathbb{C}[\alpha_1, \ldots , \alpha_n]$ in a natural way and we let $\boldsymbol{\beta}$ denote the vector of $\mathbb{L}^N$ whose entries are $\beta_1, \ldots , \beta_N$. For any seed $\s$ and any $1 \leq j \leq N$, we also set $\beta_j^{\s} := \wt(\mu_j^{\s})$ and $\boldsymbol{\beta}^{\s} := (\beta_1^{\s}, \ldots , \beta_N^{\s}) \in \mathbb{L}^N$.
 We begin with the following lemma:
 
  \begin{lem} \label{technical}
 For any seed $\s$ one has:
 $$ \frac{1}{\beta_1^{\s} \cdots \beta_N^{\s}}  =  \int_{C_{\s}} e^{-(\beta_1y_1 + \cdots + \beta_Ny_N)} dy_1 \cdots dy_N . $$
   \end{lem}
   
   \begin{proof}
   Let $C_{\s}$ be the open linear cone of $\mathbb{R}^N$ whose intersection with $\hyp$ is $\ds$. With the notations of Section~\ref{normalfan}, one has 
$C_{\s} = \bigcap_{1 \leq k \leq N} \{ \langle {\bf N}_k^{\s} , \cdot \rangle \geq 0 \}$.
   Let $\mathcal{N}_{\s}$ denote the $N \times N$ matrix whose columns are the ${\bf N}_k^{\s}, 1 \leq k \leq N$. By definition of the ${\bf N}_k^{\s}$ one has $~^t\mathcal{N}_{\s} \mathcal{M}_{\s} = Id_N$. Hence one has 
\begin{align*}
 \frac{1}{\beta_1^{\s} \cdots \beta_N^{\s}} 
& = \int_{{\mathbb{R}_{+}^{*}}^N} e^{-(\beta_1^{\s}x_1 + \cdots + \beta_N^{\s}x_N)} dx_1 \cdots dx_N \\
 & = \int_{C_{\s}} |\det(\mathcal{N}_{\s})|  e^{-(\beta_1^{\s} (~^t\mathcal{N}_{\s}y)_1 + \cdots + \beta_N^{\s}(~^t\mathcal{N}_{\s}y)_N)} dy_1 \cdots dy_N . 
 \end{align*}
 By Corollary~\ref{invert}, $| \det \left( \mathcal{M}_{\s} \right) | = 1$ and hence $| \det \left( \mathcal{N}_{\s} \right) | = 1$ as well. 
 Then for every $1 \leq j \leq N$ one has 
 $$ (~^t\mathcal{N}_{\s}y)_j = \sum_i (\mathcal{N}_{\s})_{ij} y_i $$
 and hence 
 $$ \sum_j \beta_j^{\s} (~^t\mathcal{N}_{\s}y)_j  = \sum _i \left( \sum_j  (\mathcal{N}_{\s})_{ij} \beta_j^{\s} \right) y_i 
  = \sum_i \left( \mathcal{N}_{\s} \boldsymbol{\beta}^{\s} \right)_i y_i 
  = \sum_i \left( ~^t \mathcal{M}_{\s}^{-1} \boldsymbol{\beta}^{\s} \right)_i y_i  . $$
  Then it suffices to note that for any $j$ one has 
  $$ \left( ~^t \mathcal{M}_{\s} \boldsymbol{\beta} \right)_j 
  = \langle ~^t \mathcal{M}_{\s} \boldsymbol{\beta} , {\bf e}_j \rangle
  = \langle \boldsymbol{\beta} , \bmu_j^{\s} \rangle
  = \wt ( \mu_j^{\s} )
  = \beta_j^{\s} . $$
  Thus we have proven that 
  $$ \frac{1}{\beta_1^{\s} \cdots \beta_N^{\s}}  = \int_{C_{\s}} e^{-(\beta_1y_1 + \cdots + \beta_Ny_N)} dy_1 \cdots dy_N. $$
  We conclude by performing the change of variables $ \mathbb{R}_{+}^{*} \times \ds \longrightarrow C_{\s}$ given by $(r,{\bf y}) \longmapsto r{\bf y}$. 
   \end{proof}

 \begin{proof}[Proof of Theorem~\ref{prophook}]
  If $\Aqnw$ is a cluster algebra of finite type, then all the simple objects in $\Cw$ are cluster monomials. Hence the union of the cones $C_{\s}$ for all the seeds in $\Aqnw$ is equal to the whole positive orthant ${\mathbb{R}_{+}^{*}}^N$ (up to some set of zero measure). Hence one has
\begin{align*}
  \prod_{\beta \in \Phi_{+}^{w}} \frac{1}{\beta}
  &= \int_{{\mathbb{R}_{+}^{*}}^N} e^{-(\beta_1x_1 + \cdots + \beta_Nx_N)} dx_1 \cdots dx_N \\
  &= \sum_{\s} \int_{C_{\s}} e^{-(\beta_1x_1 + \cdots + \beta_Nx_N)} dx_1 \cdots dx_N = \sum_{\s} \frac{1}{\beta_1^{\s} \cdots \beta_N^{\s}} .
  \end{align*}
  \end{proof}

 One can also state another consequence of Lemma~\ref{technical}:
 
  \begin{cor}
  Let $\s$ be any seed in $\Aqnw$. The volume of the simplex $\ds$ is given by
   $$ Vol \left( \ds \right) = \frac{1}{\prod_{1 \leq j \leq N} |\mu_j^{\s}|} . $$
   \end{cor}
   
   \begin{proof}
   Specializing the variables $\alpha_1, \ldots , \alpha_n$ to $1$ we get
 $$ \frac{1}{\prod_{1 \leq j \leq N} |\mu_j^{\s}|} =  \int_{C_{\s}} e^{-( \hgt(\beta_1) y_1 + \cdots + \hgt(\beta_N) y_N)} dy_1 \cdots dy_N . $$
   We perform the change of variables 
$$ \begin{array}{ccc}
    \mathbb{R}_{+}^{*} \times \ds & \longrightarrow & C_{\s} \\
    (r,{\bf y}) & \longmapsto & r{\bf y}
\end{array}  $$
in the right hand side. By construction, $\ds$ is included in the affine hyperplane $\hyp$ defined as $\{ \hgt(\beta_1) y_1 + \cdots + \hgt(\beta_N) y_N = 1 \}$. Hence we get $C Vol(\ds)$ where $C$ is some constant. A straightforward computation shows  that this constant is equal to $1$. 
   \end{proof}

Consequently we also get the following statement:

 \begin{cor} \label{corhook}
 Assume $w \in W$ is such that the cluster algebra $\Aqnw$ is of finite type. Then one has
$$ \frac{N!}{\prod_{\beta \in \Phi_{+}^{w}} \hgt(\beta)} = \sum_{\s} \frac{N!}{\prod_{1 \leq j \leq N} |\mu_j^{\s}|} . $$
\end{cor}

\begin{rk}
 In the general case, $\Aqnw$ can be of infinite cluster type but the sum 
 $$ \sum_{\s} \frac{N!}{\prod_{1 \leq j \leq N} |\mu_j^{\s}|} $$
 still makes sense (as the disjoint union of the simplices $\ds$ is always included in $\Delta(\Aqnw)$). We don't know if this sum still takes a remarkable form in this general situation. 
 \end{rk}

 In order to make sense of a link between Theorem~\ref{HookNakada} and Theorem~\ref{prophook}, one should take care of the conditions imposed on $w$. Theorem~\ref{HookNakada} holds under the assumption that $w$ is dominant minuscule whereas a necessary condition for Theorem~\ref{prophook} to hold is that $\Aqnw$ has to be a cluster algebra of finite type. We conjecture the following:
 
 \begin{conj} \label{conj}
 If $w \in W$ is dominant minuscule, then the cluster algebra $\Aqnw$ has a finite number of seeds. 
 \end{conj}
 
 This would imply that the equality given by Theorem~\ref{prophook}, although of different nature than Theorem~\ref{HookNakada}, holds in a larger generality. 
 
  Let us give an example where Conjecture~\ref{conj} holds. It corresponds to the example considered in {{\cite[Section 2]{Nakada}}}. 
  We consider a Lie algebra $\mathfrak{g}$ of type $A_3$ with simple roots $\Pi := \{ \alpha_1, \alpha_2, \alpha_3 \}$. We let $s_1,s_2,s_3$ denote the simple reflections of the Weyl group $W$ of $\mathfrak{g}$. 
   We set $w := s_2s_1s_3s_2$ and we choose the reduced expression ${\bf w}=(2,1,3,2)$ of $w$. It is straightforward  to check that $w$ is dominant minuscule using the criterion {{\cite[Proposition 2.3]{Stem}}}. By {{\cite[Equation 7.6]{GLS}}} (see also {{\cite[Section 2.3]{KK}}}), the subset of positive roots $\Phi_{+}^{w}$ ordered with respect to our choice of ${\bf w}$ are given by  
   $$ \Phi_{+}^{w} = \{ \alpha_2 < \alpha_1 + \alpha_2 < \alpha_2 + \alpha_3 < \alpha_1 + \alpha_2 + \alpha_3 \}  $$
In this case Nakada's colored hook formula is given in {{\cite[Section 2]{Nakada}}} and can be written as 
\begin{equation} \label{exhook}
 \begin{split}
 \frac{1}{\alpha_2} \frac{1}{\alpha_1 + \alpha_2} \frac{1}{\alpha_2 + \alpha_3} \frac{1}{\alpha_1 + \alpha_2 + \alpha_3} & = \frac{1}{\alpha_2} \frac{1}{\alpha_1 + \alpha_2} \frac{1}{\alpha_1 + \alpha_2 + \alpha_3} \frac{1}{\alpha_1 + 2 \alpha_2 + \alpha_3} \\
     & \qquad  + \frac{1}{\alpha_2} \frac{1}{\alpha_2 + \alpha_3} \frac{1}{\alpha_1 + \alpha_2 + \alpha_3} \frac{1}{\alpha_1 + 2 \alpha_2 + \alpha_3} . 
  \end{split}
\end{equation}
On the other hand, the cluster algebra $\Aqnw$ is of rank $4$ and with the notations of Section~\ref{remindKK}, $J_{ex} = \{1\}$ and $J_{fr} = \{2,3,4\}$. Thus there is only one mutation direction and hence by the involutivity of cluster mutations there are exactly two seeds in $\Aqnw$. Each of these two seeds contains one unfrozen variable and three frozen variables. The chosen reduced expression of $w$ corresponds to the restriction on $\Phi_{+}^{w}$ of the ordering $2<1<3$ on the $A_3$ Dynkin diagram. Thus Theorem~\ref{initpara} gives the dominant words $\mu_i^{\s} , i=1, \ldots ,4$ for the seed $\s = \mathcal{S}^{\bf w}$. Here we only need their weights, which are given by
$$ \wt(\mu_1) = \alpha_2  \quad \wt(\mu_2) = \alpha_1 + \alpha_2 \quad \wt(\mu_3) = \alpha_2 + \alpha_3 \quad \wt(\mu_4) = \alpha_1 + 2 \alpha_2 + \alpha_3 . $$
The exchange matrix of the seed  $\mathcal{S}^{\bf w}$ is given by $B =~^t (0,1,1,-1)$
and thus the cluster variable $x'_1$ obtained after performing the mutation at $x_1$ of the seed $\mathcal{S}^{\bf w}$ is given by 
$$ x'_1 = \frac{1}{x_1} \left( x_2x_3 + x_4 \right) . $$
 The corresponding dominant word $\mu'_1$ has weight $wt(\mu'_1) = \alpha_1 + \alpha_2 + \alpha_3$.
Thus the sum in the right hand side of Equation~\eqref{eqprophook} is 
$$ \frac{1}{\alpha_2} \frac{1}{\alpha_1 + \alpha_2} \frac{1}{\alpha_2 + \alpha_3} \frac{1}{\alpha_1 + 2 \alpha_2 + \alpha_3} + \frac{1}{\alpha_1 + \alpha_2 + \alpha_3} \frac{1}{\alpha_1 + \alpha_2} \frac{1}{\alpha_2 + \alpha_3} \frac{1}{\alpha_1 + 2 \alpha_2 + \alpha_3} . $$
In this example one can check by a straightforward calculation that this rational function is equal to 
$$  \frac{1}{\alpha_2} \frac{1}{\alpha_1 + \alpha_2} \frac{1}{\alpha_2 + \alpha_3} \frac{1}{\alpha_1 + \alpha_2 + \alpha_3} $$
which is exactly the statement of Theorem~\ref{prophook}. 

\begin{rk}
The sums of rational functions on the right hand side of Equations~\eqref{eqNakada}  and~\eqref{eqprophook} are a priori of a very different combinatorial natures. For instance these sums do not have the same number of terms in general. Moreover when specializing the $\alpha_i$ to $1$, the terms in the right hand side of Equation~\eqref{eqNakada} all take the same value $1/N!$. On the contrary the value taken by the term indexed by a seed $\s$ in Equation~\eqref{eqprophook} is essentially the volume of the simplex $\ds$, which is not the same for every seed. However, it turns out (even in more complicated examples) that these two different expressions take rather similar forms. This might suggest closer connections between the combinatorics of fully-commutative elements of Weyl groups and cluster theory. 
\end{rk}

 \begin{rk} 
The rational functions on the right hand-side of Equation~\eqref{eqNakada} also appeared in the recent work of Baumann-Kamnitzer-Knutson \cite{BKK}. They are related with the definition of Duistermaat-Heckmann measures used to compare various bases in $\Aqn$. In an appendix of the same work \cite{BKK}, Dranowski-Kamnitzer-Morton-Ferguson used these tools to prove that the Mirkovi\'c-Vilonen basis and the dual semicanonical basis are not the same.
 \end{rk}

\Address

\end{document}